\documentclass[reqno,english]{amsart}
\usepackage{bbding,textcomp,amsfonts,amsthm,amsmath,mathrsfs}
\usepackage{amssymb}
\usepackage[pdftex,colorlinks,backref=page,citecolor=blue]{hyperref}
\hypersetup{pdfpagemode=UseNone,pdfstartview={XYZ null null 1.00}}
\usepackage{multirow, url}
\usepackage[utf8x]{inputenc}
\usepackage{comment}
\usepackage[T1]{fontenc}
\usepackage{ae,aecompl}
\usepackage{enumerate}
\usepackage{ytableau}
\usepackage{bbm}
\usepackage{centernot}

\newtheorem{definition}{Definition}
\newtheorem{theorem}{Theorem}
\newtheorem{conjecture}{Conjecture}
\newtheorem{corollary}[theorem]{Corollary}
\newtheorem{lemma}[theorem]{Lemma}
\newtheorem{proposition}[theorem]{Proposition}
\newtheorem{observation}[theorem]{Observation}
\newtheorem{claim}{Claim}

\newtheorem{case}{Case}

\numberwithin{subcase}{case}

\newcommand{\ex}{\textrm{\upshape{ex}}}

\newcommand{\Wr}{\textrm{\upshape{Wr}}}
\newcommand{\He}{\textrm{\upshape{He}}}
\renewcommand{\mod}{\ \textrm{\upshape{mod}}\ }

\newcommand{\By}[2]{\overset{\mbox{\tiny{#1}}}{#2}}
\newcommand{\ByRef}[2]{   \By{\eqref{#1}}{#2} }

\newcommand{\eqByRef}[1]{ \ByRef{#1}{=} }

\newcommand{\leByRef}[1]{ \ByRef{#1}{\le} }

\PrerenderUnicode{\unichar{355}}

\title{The irreducibility of some Wronskian Hermite polynomials}

\author{Codru\unichar{355} Grosu}
\address{Google Z{\"u}rich, Brandschenkestrasse 110, Z{\"u}rich, Switzerland}
\email{grosu.codrut@gmail.com}

\author{Corina Grosu}
\address{Department of Applied Mathematics, Politehnica University of Bucharest, Splaiul Independentei 303, Bucharest, Romania}
\email{grosu\_c1990@yahoo.com}

\begin{document}

\maketitle

\begin{abstract}
We study the irreducibility of Wronskian Hermite polynomials labelled by partitions. It is known that these polynomials factor as a power of $x$ times a remainder polynomial. We show that the remainder polynomial is irreducible for the partitions $(n, m)$ with $m \leq 2$, and $(n, n)$ when $n+1$ is a square.

Our main tools are two theorems that we prove for all partitions. The first result gives a sharp upper bound for the slope of the edges of the Newton polygon for the remainder polynomial. The second result is a Schur-type congruence for Wronskian Hermite polynomials.

We also explain how irreducibility determines the number of real zeros of Wronskian Hermite polynomials, and prove Veselov's conjecture for partitions of the form $(n, k, k-1, \ldots, 1)$.
\end{abstract}

\section{Introduction}

The study of Wronskians of one variable orthogonal polynomials $\{Q_n(x)\}_{n \geq 0}$ or determinants with entries one variable orthogonal polynomials (Hankelians and Tur\'anians) was influenced by their applications in obtaining solutions to differential equations in mathematical physics (\cite{VeinDale}) and probability theory (\cite{KarlinSzego}). In particular, Wronskians of classical polynomials (ultraspherical, Laguerre and Hermite polynomials) were investigated for the range of parameters and the domain in which Tur\'an’s inequality holds (\cite{KarlinSzego}):
\begin{center}
$Q_n^2(x)-Q_{n-1}(x)Q_{n+1}(x)\geq 0$
\end{center}
This inequality plays an important role in the study of birth-and-death processes (\cite{Chihara}, \cite{KarlinMcGregor}). A similar analysis was done for the so called augmented Wronskians by Karlin and Szeg{\H o} (\cite{KarlinSzego}), namely 
\begin{center}
\begin{equation*}
\Wr[Q_k,Q_n,Q_{n+1}, \ldots, Q_{n+\ell}] =
\begin{vmatrix}
Q_k & Q_n & Q_{n+1} & \cdots & Q_{n+\ell} \\ 
Q_k' & Q_n' & Q_{n+1}' & \cdots & Q_{n+\ell}' \\
\vdots & \vdots & \ddots & \vdots \\
Q_k^{(\ell+1)} & Q_n^{(\ell+1)} & Q_{n+1}^{(\ell+1)} & \cdots & Q_{n+\ell}^{(\ell+1)}
\end{vmatrix}.
\end{equation*}
\end{center}
We have to remark here that some of the results obtained in these studies concern the existence and multiplicity of the real roots of the respective Wronskians. Since this type of analysis involves higher order Wronskians and their derivatives, recurrence relations like Jacobi’s identity for Wronskians
and the Laplace expansion of determinants are the main tools for extending the results obtained for low order Wronskians to higher ones (\cite{VeinDale}). 
More recently, Wronskians of orthogonal polynomials appeared in the study of exceptional orthogonal polynomials (\cite{Duran}, \cite{UllateMilson}, \cite{OdakeSasaki}, \cite{Quesne}), again with an interest towards their applications in mathematical physics (Darboux-Backlund transformations of solvable quantum potentials). Since exceptional orthogonal polynomials involve quotients of two Wronskians, the existence and location of their zeros plays an important role in this analysis (\cite{Duran}, \cite{KuijlaarsWilson15}).

As shown by Oblomkov (\cite{Oblomkov}), Wronskians of Hermite polynomials characterize rational potentials of monodromy-free Schr{\H o}dinger operators that grow as $x^2$ at infinity. They also provide rational solutions to the fourth Painlev{\'e} equation.

In this paper we will be interested in the irreducibility of Wronskians of Hermite polynomials. This is motivated by applications to the set of zeros which we explain later below.
 
Let $\{H_n(x)\}_{n \geq 0}$ be the classical Hermite polynomials, solutions to the equation $y''(x)-2xy'(x)+2ny(x)=0$. Furthermore, let $\{\He_n(x)\}_{n \geq 0}$ be the probabilistic Hermite polynomials, solutions to the equation $y''(x)-xy'(x)+ny(x)=0$. The relation between the two is given by $\He_n(x) = 2^{-\frac{n}{2}}H_n(\frac{x}{\sqrt{2}})$. We shall mostly work with $\He_n(x)$, which is integral and monic.

Schur (\cite{Schur2}) showed that $\He_{2n}(x)$ is irreducible for every $n \neq 1$, and similarly $\frac{\He_{2n+1}(x)}{x}$ is irreducible for every $n \geq 0$. Following the path opened by Schur, the irreducibility of other classes of orthogonal polynomials was studied. Filaseta and Trifonov (\cite{FilasetaTrifonov}) proved that all Bessel polynomials are irreducible. Schur (\cite{Schur30}) showed that the classical Laguerre polynomials $L_n^{(0)}(x)$ are irreducible. This result was extended to the generalized Laguerre polynomials $L_n^{(\alpha)}(x)$ for many values of $\alpha$ (see for example \cite{Filaseta02}, \cite{Shorey16}).

 We will define Wronskians of Hermite polynomials in terms of partitions. Let $\lambda = (\lambda_1 \geq \lambda_2 \geq \ldots \geq \lambda_r)$ be any partition. We define the \textit{degree sequence of $\lambda$} as $n_\lambda := (\lambda_r, \lambda_{r-1}+1, \ldots, \lambda_1+r-1)$. Furthermore, let
\begin{equation*}
\Delta(x_1, x_2, \ldots, x_r) := \det[x_i^{j-1}]_{1 \leq i, j \leq r} = \prod_{j>i}(x_j-x_i)
\end{equation*}
be the Vandermonde determinant, with $\Delta(x_1) := 1$.

\begin{definition}[Wronskian Hermite polynomial]
\label{def:wronskian}
For any partition $\lambda \vdash n$ we define the Wronskian Hermite polynomial associated to $\lambda$ as
\begin{equation}
\He_{\lambda}(x) := \frac{\Wr[\He_{n_1}(x), \He_{n_2}(x), \ldots, \He_{n_r}(x)]}{\Delta(n_\lambda)},
\end{equation}
where $n_\lambda = (n_1, n_2, \ldots, n_r)$ is the degree sequence of $\lambda$.
\end{definition}
Then $\He_\lambda(x)$ is a monic polynomial of degree $n$.

Recently, substantial progress was made in understanding the polynomials $\He_\lambda(x)$. A recurrence relation for $\He_\lambda(x)$ was obtained by Bonneux and Stevens in \cite{BonneuxStevens}. The authors also derived many interesting properties of these polynomials. Later in \cite{BonneuxAppell}, the polynomials $\He_\lambda(x)$ were shown to have integer coefficients, and finally in \cite{BonneuxCoefficients} a formula for the coefficients was obtained. Unfortunately, the formula depends on the characters of the symmetric group, and hence can not be reduced to a simple form.

For a partition $\lambda$, define $d_\lambda := p - q$, where $p$, respectively $q$, is the number of odd, respectively even, integers in the degree sequence $n_\lambda$. From Theorem~3.1, \cite{BonneuxCoefficients}, the polynomial $\He_\lambda(x)$ can be decomposed as $x^{\binom{d_\lambda+1}{2}}R_\lambda(x)$, where $R_\lambda(0) \neq 0$. $R_\lambda(x)$ is called the \textit{remainder polynomial}.

The result of Schur (\cite{Schur2}) can now be restated as $R_n(x)$ is irreducible, for all $n \neq 2$ (for $n=2$ we have $\He_2(x) = R_2(x) = x^2 - 1$ which is reducible). Therefore we can hope that $R_\lambda(x)$ is irreducible for all partitions $\lambda \neq (2)$. Unfortunately this is false. A computer search shows that already for $n = 9$ there exists partitions $\lambda \vdash n$ such that $R_\lambda(x)$ is reducible. For $n=9$, these are the partitions $(6, 1, 1, 1), (5, 1, 1, 1, 1)$ and $(4, 1, 1, 1, 1, 1)$.

We carried out a computer search for partitions with reducible remainder polynomial (\cite{GrosuGit}). We divided the search by the length of the partition, and for each length $\ell$, we checked all partitions $\lambda \vdash n$ with $n \leq N$ and $\ell$ parts. The value of $N$ and the list of partitions with reducible $R_\lambda(x)$ are displayed in the following table.
\begin{center}
\begin{tabular}{ |c|c|c| }
\hline
Length & N & $R_\lambda(x)$ reducible\\
\hline
2 & 1000 & \\
\hline
3 & 250 & (7, 3, 1) \\
\hline
4 & 150 & (6, 1, 1, 1), (6, 3, 2, 1), (6, 5, 3, 3)\\
\hline
5 & 110 & (5, 1, 1, 1, 1), (5, 3, 2, 1, 1), (5, 3, 3, 1, 1), (5, 4, 4, 3, 1)\\
\hline
\end{tabular}
\end{center}
This suggests that for fixed length, there are only finitely many reducible examples, while for length $2$, there are no examples at all. As evidence for the latter statement, we were able to show the following.
\begin{theorem}
\label{thm:main1}
The polynomial $R_{n, n}(x)$ is irreducible in $\mathbb{Z}[x]$ if $n+1$ is a square.
\end{theorem}

\begin{theorem}
\label{thm:main2}
The polynomial $R_{n, 1}(x)$ is irreducible in $\mathbb{Z}[x]$ for any $n \geq 1$.
\end{theorem}

\begin{theorem}
\label{thm:main3}
The polynomial $R_{n, 2}(x)$ is irreducible in $\mathbb{Z}[x]$ for any $n \geq 2$.
\end{theorem}

The irreducibility of $R_\lambda(x)$ would have two important applications.

The first application is to the multiplicity of zeros of $\He_\lambda(x)$. Veselov conjectured that the zeros of Wronskians of Hermite polynomials are always simple, except possibly at the origin:
\begin{conjecture}[\cite{FelderHemeryVeselov12}]
\label{conj:veselov}
For any positive integers $n_1, n_2, \ldots, n_r$, the Wronskian $\Wr[H_{n_1}(x), H_{n_2}(x), \ldots, H_{n_r}(x)]$ has simple zeros, except possibly for $x=0$.
\end{conjecture}
Conjecture~\ref{conj:veselov} is known in a few cases, but in general it is still open. We give a short overview of known results in Section~\ref{sec:applications}.

Because $\He_n(x)$ is a rescaling of $H_n(x)$, Conjecture~\ref{conj:veselov} is equivalent to the statement that $\He_\lambda(x)$ has simple zeros, except possibly for $x=0$. As $\He_\lambda(x) = x^{\binom{d_\lambda+1}{2}}R_\lambda(x)$, this would be implied by the irreducibility of $R_\lambda(x)$. 

In fact, the irreducibility of the classical Hermite polynomials $H_n(x)$ was used in \cite{FerreroUllate15} to show that Conjecture~\ref{conj:veselov} holds for $\Wr[H_n(x), H_m(x)]$ for any $n$ and $m$. Thus for partitions of length $2$, $R_\lambda(x)$ has only simple non-zero roots.

The second application is to the number of real zeros of $\He_\lambda(x)$. In \cite{FerreroUllate15}, the authors study the number of real roots of a Wronskian of eigenfunctions of Schr{\H o}dinger's equation:
\begin{equation}
\label{eq:schrodinger}
-\varphi''(x) + V(x)\varphi(x) = E\varphi(x),
\end{equation}
where $\lim_{x\rightarrow \pm \infty} \varphi(x) = 0$. It is known that the Hermite functions $\varphi_n(x) := e^{-\frac{x^2}{2}}H_n(x)$ verify \eqref{eq:schrodinger} for $V(x) := x^2$ and $E := 2n+1$. For symmetric potentials $V(x)$ such as $x^2$, and for a \textit{semi-degenerate} sequence of eigenfunctions $\{\varphi_n\}_{n \geq 0}$, Theorem~$1.4$ of \cite{FerreroUllate15} gives a formula for the number of real roots of $\Wr[\varphi_{n_1}, \varphi_{n_2}, \ldots, \varphi_{n_r}]$. However, it is not known if the Hermite functions form a semi-degenerate sequence, and in fact, this question is closely related to irreducibility.  

We give below a definition of semi-degeneracy that is weaker than the one stated in \cite{FerreroUllate15}.
\begin{definition}[Semi-degenerate sequence]
\label{def:semidegenerate}
Let $\{\varphi_n\}_{n \geq 0}$ be a sequence of eigenfunctions of Schr{\H o}dinger's equation. Let $n_1 < n_2 < \ldots < n_r$ be an increasing sequence of non-negative integers. We call this sequence semi-degenerate if the following two conditions hold:
\begin{enumerate}[(i)]
\item For any $1 \leq i < j \leq r$, the Wronskians $\Wr[\varphi_{n_1}, \ldots, \varphi_{n_i}]$ and $\Wr[\varphi_{n_1}, \ldots, \varphi_{n_i}, \varphi_{n_j}]$ have at most the root $x = 0$ in common.
\item For any $1 \leq i < j \leq r$, the Wronskians $\Wr[\varphi_{n_1}, \ldots, \varphi_{n_{i-1}}, \varphi_{n_i}]$ and $\Wr[\varphi_{n_1}, \ldots, \varphi_{n_{i-1}}, \varphi_{n_j}]$ have at most the root $x = 0$ in common.
\end{enumerate}
\end{definition}

By examining the proof of Theorem~$1.4$ in \cite{FerreroUllate15} it turns out that Definition~\ref{def:semidegenerate} is enough to imply the statement. Therefore the following holds.
\begin{theorem}[Theorem~$1.4$, \cite{FerreroUllate15}]
\label{thm:gomez}
Let $\varphi_n(x) := e^{-\frac{x^2}{2}}H_n(x)$ be the Hermite functions, solutions to the equation $-\varphi_n''(x) + x^2\varphi_n(x) = (2n+1)\varphi_n(x)$. If $\lambda$ is a partition with semi-degenerate degree sequence $(n_1, n_2, \ldots, n_r)$ then the Wronskian $\Wr[\varphi_{n_1}(x), \ldots, \varphi_{n_r}(x)]$ has
\begin{enumerate}[(i)]
\item a root at $x = 0$ of multiplicity $\frac{d_\lambda(d_\lambda + 1)}{2}$,
\item all non-zero real roots are simple, out of which
\begin{equation*}
\frac{1}{2}\left(\sum_{i=1}^r (-1)^{i-1}\lambda_i - \frac{|d_\lambda+(r-2\left\lfloor\frac{r}{2}\right\rfloor)|}{2}\right)
\end{equation*}
are positive,
\item the same number of negative and positive real roots, due to symmetry.
\end{enumerate}
\end{theorem}

Theorem~\ref{thm:gomez} can be used together with our irreducibility results to show the following.
\begin{corollary}
\label{cor:num_real_roots}
Let $\lambda = (\lambda_1, \lambda_2, \lambda_3)$ be a partition with $1 \leq \lambda_3 \leq 2$. Then $\He_\lambda(x)$ has
\begin{equation*}
\lambda_1 - \lambda_2 + \lambda_3 - \frac{|d_\lambda + 1|}{2}
\end{equation*}
non-zero real roots, all of which are simple.
\end{corollary}

In fact Theorem~\ref{thm:gomez} has further applications. For example, it can be used to establish Veselov's conjecture in the following new instance.
\begin{proposition}
\label{prop:staircase}
Let $\lambda=(n,k,k-1,\ldots,1)$ with $n \geq k \geq 1$.
\begin{enumerate}[(i)]
\item If $n - k$ is odd, then all the non-zero roots of $\He_{\lambda}(x)$ are simple and real and their number is $n-k-1$.
\item If $n - k$ is even, then $\He_{\lambda}(x)$ has, apart from $0$,  only simple roots, from which $n - k$ are real and $2k$ are complex non-real.
\end{enumerate}
\end{proposition}

In view of Proposition~\ref{prop:staircase}, one may ask if $R_\lambda(x)$ is irreducible for $\lambda=(n,k,k-1,\ldots,1)$. This is false, the smallest counterexample being $\lambda=(6, 3, 2, 1)$. In this case, $R_\lambda(x) = x^2 - 9$.

Although we were able to establish irreducibility only in a few special cases, the main tools that we use are two theorems that we prove for all partitions $\lambda$. The first result gives a sharp upper bound for the slope of the edges of the Newton polygon for $R_\lambda(x)$. We state this result in terms of the $2$-core of a partition (see Section~\ref{subsec:cores} for definition).
\begin{theorem}
\label{thm:slope}
Let $\lambda \vdash n$ be a partition with $2$-core of size $s$, and write $s = \binom{m+1}{2}$ with $m \geq 0$. If $p > \max\{2, 2m-1\}$ is a prime number that does not divide $\Delta(n_\lambda)$, then the slope of the right-most edge of the Newton polygon for $R_\lambda(x)$ with respect to $p$ is
\begin{enumerate}[(i)]
\item strictly less than $\frac{1}{p-1}$, if $m=0$. If further $n < p^2$ then the slope is at most $\frac{1}{p}$.
\item at most $\frac{1}{p-(2m-1)}$, if $m \geq 1$.
\end{enumerate}
Moreover, the upper bound for $m \geq 1$ is tight, while any upper bound for $m = 0$ must be at least $\frac{1}{p}$.
\end{theorem}

Our second main tool is a Schur-type congruence for $\He_\lambda(x)$.
\begin{theorem}
\label{thm:hermitemodp}
Let $\lambda$ be a partition with $r$ parts and $m \geq 3$ an odd integer such that $m$ and $\Delta(n_\lambda)$ are coprime. Then the integers $n_{\lambda, i} \mod m, 1 \leq i \leq r$, are pairwise distinct and form the degree sequence of a partition $\mu$.

Furthermore,
\begin{equation*}
\He_\lambda(x) \equiv x^{|\lambda| - |\mu|}\He_\mu(x)\mod m.
\end{equation*}
\end{theorem}

For the irreducibility of $R_{n, 2}(x)$ we also need an upper bound for the modulus of real or purely imaginary zeros. We established this bound in a separate paper.
\begin{lemma}[\cite{Grosu2019}]
\label{lem:bound_real_roots}
Let $\lambda \vdash n$. If $z$ is a real or purely imaginary root of $\He_\lambda(x)$ then $|z| \leq x_n$, where $x_n$ is the largest root of $\He_n(x)$.
\end{lemma}

We hope these tools will be useful in showing irreducibility in other cases as well.

The rest of this paper is organized as follows. In Section~\ref{sec:notation}, we gather the notation used throughout the paper, as well as several auxiliary results that we will need. In Section~\ref{sec:degrees} we solve an extremal problem for character degrees. This result is an important ingredient for the proof of Theorem~\ref{thm:slope}. In Section~\ref{sec:slope}, we obtain an upper bound for slope of the edges of the Newton polygon by proving Theorem~\ref{thm:slope}. In Section~\ref{sec:schur}, we determine the polynomials $\He_\lambda(x) \mod m$ under the conditions of Theorem~\ref{thm:hermitemodp}. In Section~\ref{sec:main1} we prove Theorem~\ref{thm:main1}, followed by Theorem~\ref{thm:main2} in Section~\ref{sec:main2}. The proof of Theorem~\ref{thm:main3} occupies Section~\ref{sec:main3}. Finally, in Section~\ref{sec:applications} we prove Corollary~\ref{cor:num_real_roots} and Proposition~\ref{prop:staircase}.

\section{Notation and auxiliary results}
\label{sec:notation}

In this section we gather the notation we use throughout the paper, as well as several results we will need later. For the representation theory of the symmetric group we take as main reference \cite{JamesReprTheory}.

\subsection{Partitions and characters}

If $n \geq 0$ is an integer, a partition $\lambda$ of $n$, denoted $\lambda \vdash n$, is a sequence of nonnegative integers $\lambda_1 \geq \lambda_2 \geq \ldots \geq \lambda_r \geq 0$ such that $\sum_{i=1}^r\lambda_i = n$. We denote $|\lambda|=n$ and call $\ell(\lambda) := r$ the length of the partition $\lambda$. We say that $\lambda_i$ are the \textit{parts} of the partition.

Note that we deviate from the standard definition by allowing parts of size $0$. This will simplify many of our statements and proofs.

We shall frequently use the notation $(\lambda_1, \lambda_2, \ldots, \lambda_r)$ for $\lambda$. Sometimes we will also use the notation $\lambda = 0^{r_0}1^{r_1}2^{r_2}3^{r_3}\ldots$, meaning that the partition $\lambda$ has $r_i$ parts of size $i$.

The \textit{degree sequence} of the partition $\lambda$ is defined as $n_\lambda := (\lambda_r, \lambda_{r-1}+1, \ldots, \lambda_1+r-1)$. All the integers in $n_\lambda$ are distinct and non-negative, so $n_\lambda$ can be regarded as a set. Furthermore, as we allow partitions to have parts of size $0$, any set of $r$ non-negative integers is the degree sequence of a unique partition of length $r$. 

The \textit{Ferrers diagram} of a partition $\lambda$ of length $r$ is $D_\lambda = \{(i, j): 1 \leq i \leq r, 1 \leq j \leq \lambda_i\}$. This can be represented as a collection of unit squares arranged in rows, with the $i$-th row having $\lambda_i$ squares. For example,

\begin{equation*}
\ytableausetup{centertableaux, nosmalltableaux}
D_{(4, 4, 2, 1)} = \ydiagram{4,4,2,1}
\end{equation*}

A $q$-hook is any connected set of squares in $D_\lambda$ of size $q$, whose removal produces a valid partition. Then any $q$-hook contains only \textit{border squares}: these are squares $(a, b) \in D_\lambda$ such that either $(a, b+1), (a+1,b)$ or $(a+1, b+1)$ are not in $D_\lambda$. If $q=2$, then the only possible $2$-hooks are two adjacent squares at the end of a row of $\lambda$, which we say is an \textit{horizontal $2$-hook}, or two adjacent squares at the end of a column, which we say is a \textit{vertical $2$-hook}.

If $R$ is a $q$-hook, we let $\lambda \setminus R$ denote the partition obtained by removing it. By adding $0$ if necessary, we keep the length of $\lambda \setminus R$ the same as $\lambda$. 

We can define a partial order on the set of partitions by saying that $\mu \leq \lambda$ if $\ell(\lambda) = \ell(\mu)$ and $\lambda_i \geq \mu_i$ for all $i \leq \ell(\lambda)$. If $\mu \leq \lambda$ then $\lambda/\mu$ denotes the skew shape $D_{\lambda/\mu} := D_\lambda \setminus D_\mu$. We further write $\mu \leq_k \lambda$ if $\mu$ can be obtained from $\lambda$ by removing $k$ $2$-hooks.

We denote by $\chi^\lambda$ the irreducible character of $S_n$ associated to $\lambda$. Let $F_\lambda := \chi^\lambda(1)$ be the degree of the irreducible representation. Then this is given by the formula (see \cite{JamesReprTheory}, 2.3.22):
\begin{equation}
\label{eq:hnfactorial}
F_\lambda = \frac{|\lambda|!}{H(\lambda)}, \textrm{ where }H(\lambda):=\frac{n_{\lambda, 1}!n_{\lambda, 2}!\ldots n_{\lambda, \ell(\lambda)}!}{\Delta(n_\lambda)}.
\end{equation}

The character values can be computed using the Murnaghan-Nakayama formula. If $\lambda, \mu \vdash n$ and $\mu$ has a part of size $q$, then
\begin{equation}
\label{eq:murnaghan-nakayama}
\chi^\lambda(\mu) = \sum_{\textrm{$R$ $q$-hook}}(-1)^{ht(R)}\chi^{\lambda\setminus R}(\mu - q),
\end{equation}
where $ht(R)$ is the \textit{height} of $R$, defined as one less than the number of rows spanned by $R$, and $\mu - q$ is the partition obtained from $\mu$ by removing a part of size $q$.

We are going to drop the subscript or superscript $\lambda$ if it is clear from the context.

\subsection{$2$-cores and $2$-quotients of partitions}
\label{subsec:cores}

In this section we will define the $2$-core and $2$-quotient of a partition. These can be defined more generally for any natural number $q$. See \cite{JamesReprTheory}, Chapter 2.7, for the generalization and proofs. A detailed exposition of the theory is also given in \cite{Morotti2011}.

It is useful to introduce first the notion of $2$-abacus.

\begin{definition}[$2$-abacus]
The $2$-abacus consists of $2$ vertical runners indexed from left with $0$ and $1$. The first runner contains the positions $0, 2, 4, \ldots$, while the second runner contains the positions $1, 3, 5, \ldots$, starting from the top and moving downwards.
\end{definition}
So the $2$-abacus looks like this
\begin{equation*}
\begin{matrix}
0 & 1 \\
2 & 3 \\
4 & 5 \\
\vdots & \vdots
\end{matrix}
\end{equation*}

Now suppose we are given a partition $\lambda$ with degree sequence $n_\lambda = (n_1, n_2, \ldots, n_r)$. Then we place a bead on the $2$-abacus on each of the numbers $n_r, n_{r-1}, \ldots, n_1$. For example, if $\lambda = (3, 3, 2)$ then $n_{\lambda}=(2, 4, 5)$, and the $2$-abacus for $\lambda$ is given by
\begin{equation*}
\begin{matrix}
0 & 1 \\
\textcircled{2} & 3 \\
\textcircled{4} & \textcircled{5} \\
\vdots & \vdots
\end{matrix}
\end{equation*}

Conversely, a $2$-abacus with $r$ beads placed on it defines the degree sequence of a partition: we take the location of the beads as the values $n_1, n_2, \ldots, n_r$. This means there is a bijective correspondence between partitions and $2$-abaci with beads.

The $2$-abacus is a useful instrument for visualizing the removal of $2$-hooks from a partition. It is easy to see that a $2$-hook corresponds to a bead with an empty space above it on the runner. Removing the $2$-hook is the same as moving the bead up one space on its runner. If we recursively remove $2$-hooks until none exists, we will always end up with the same partition, the one obtained by moving all beads as high as possible on the $2$-abacus. We shall define this as the $2$-core of the initial partition. 

\begin{definition}[$2$-core, $2$-weight]
If $\lambda$ is any partition, the $2$-core $\bar{\lambda}$ is the partition obtained by removing from $\lambda$ the maximum possible number of $2$-hooks. The number of hooks we need to remove to obtain $\bar{\lambda}$ is called the $2$-weight of $\lambda$ and is denoted by $w_2(\lambda)$.
\end{definition}
In our example with $\lambda = (3, 3, 2)$, we can move the beads to locations $0, 1$ and $2$. So the $2$-core will have degree sequence $(2, 1, 0)$, in other words, it will be the partition $(0, 0, 0)$.

Note that the $2$-weight is given by the formula $w_2(\lambda) = \frac{|\lambda| - |\bar{\lambda}|}{2}$. From the Murnaghan-Nakayama formula one can obtain the following.

\begin{lemma}[\cite{JamesReprTheory}, Corollary~$2.7.33$]
\label{lem:overweight}
Let $\lambda, \mu \vdash n$ be partitions such that $\mu$ has $k$ $2$-cycles, with $k > w_2(\lambda)$. Then $\chi^\lambda(\mu)=0$.
\end{lemma}

\begin{definition}[$2$-quotient]
The $2$-quotient of a partition $\lambda$ is a pair of partitions $(\pi_0, \pi_1)$, where $\pi_0$ has the degree sequence given by the beads on the first runner, and $\pi_1$ has the degree sequence given by the beads on the second runner of the $2$-abacus for $\lambda$. 
\end{definition}
In our example, the first runner can be interpreted as a $2$-abacus with beads on positions $1$ and $2$. Hence this corresponds to the partition $\pi_0 = (1, 1)$. The second runner can be interpreted as a $2$-abacus with a single bead on position $2$. Hence this gives the partition $\pi_1 = (2)$.

Because we allow $0$ elements in partitions, the $2$-quotient uniquely determines the starting partition $\lambda$. If we would drop the zeros, we would no longer know the lengths of $\pi_0$ and $\pi_1$. This information can be recovered from $\bar{\lambda}$, which is why some authors say that the $2$-core and $2$-quotient uniquely determine the partition, and not just the $2$-quotient alone.

\begin{definition}[Natural numbering]
Let $X$ be a collection of beads on a $2$-abacus. Then the natural numbering of $X$ is given by numbering the elements of $X$ increasingly starting from $1$, according to their position on the abacus.
\end{definition}

Because $\bar{\lambda}$ is obtained from $\lambda$ by sliding beads up on the $2$-abacus, and both partitions have the same length, we can identify the beads on the $2$-abacus for $\bar{\lambda}$ with the beads on the $2$-abacus for $\lambda$, based on their order on each runner. Let $X$ be the set of these beads. Then $X$ has a natural numbering induced by the $2$-abacus for $\lambda$, and a (perhaps different) natural numbering induced by the $2$-abacus for $\bar{\lambda}$. We let $\delta_2(\lambda)$ be the sign of the permutation that changes these $2$ numberings into one another (see also \cite{JamesReprTheory}, the discussion after $2.7.20$).

It can be shown (\cite{JamesReprTheory}, $2.7.32$) that removing a $2$-hook from $\lambda$ is the same as removing a corner cell (or $1$-hook) from the Ferrers diagram of either $\pi_0$ or $\pi_1$. From this the following follows.

\begin{theorem}[\cite{JamesReprTheory}, Corollary~$2.7.33$]
\label{thm:value2weight}
Let $\lambda \vdash n$ be a partition with $2$-quotient $(\pi_0, \pi_1)$. Let $m$ be the $2$-weight of $\lambda$. Then
\begin{equation}
\label{eq:char2m}
\chi^\lambda(2^m1^{n-2m}) = \delta_2(\lambda) \binom{|\pi_0|+|\pi_1|}{|\pi_0|}F_{\pi_0}F_{\pi_1}F_{\bar{\lambda}}.
\end{equation}
\end{theorem}

For our purposes we will also need information about $\chi^\lambda(2^k1^{n-2k})$ when $k < w_2(\lambda)$.

Let $\mu \leq_k \lambda$ be any partition that can be obtained from $\lambda$ by removing $k$ $2$-hooks. Then $\mu$ is obtained from $\lambda$ by sliding up beads on the $2$-abacus $k$ times. As above, we can define a permutation that sends the natural numbering of the beads for $\lambda$ into the natural numbering of the beads for $\mu$. Let $\delta_2(\lambda, \mu)$ be the sign of this permutation. Furthermore, let $F_{2,\lambda/\mu}$ be the number of ways we can obtain $\mu$ from $\lambda$ by removing $2$-hooks.

\begin{theorem}
\label{thm:valuelessthan2weight}
Let $\lambda \vdash n$ and $k \leq w_2(\lambda)$. Then
\begin{equation}
\chi^\lambda(2^k1^{n-2k}) = \sum_{\mu \leq_k \lambda}\delta_2(\lambda, \mu) F_{2,\lambda/\mu} F_{\mu}.
\end{equation}
\end{theorem}
\begin{proof}
Let $\mathcal{S}$ be the set of sequences of $2$-hooks $\xi_1, \xi_2, \ldots, \xi_k$ of length $k$ which can be recursively removed from $\lambda$. Therefore $\xi_1$ is a $2$-hook in $\lambda$, $\xi_2$ is a $2$-hook in $\lambda \setminus \xi_1$, and so on. The Murnaghan-Nakayama rule tells us that
\begin{equation}
\label{eq:sum_of_xi}
\chi^\lambda(2^k1^{n-2k}) = \sum_{\xi_1, \xi_2, \ldots, \xi_k \in \mathcal{S}}\left(\prod_{i=1}^k (-1)^{ht(\xi_i)}\right)F_{\lambda \setminus \xi_1 \setminus \xi_2 \ldots \setminus \xi_k},
\end{equation}
where $ht(\xi_i)$ is the height of the $2$-hook $\xi_i$.

Let $\mu \leq_k \lambda$ be any partition that can be obtained from $\lambda$ by removing $k$ $2$-hooks. Let $\xi_1, \xi_2, \ldots, \xi_k \in \mathcal{S}$ arbitrary with $\lambda \setminus \xi_1 \setminus \xi_2 \ldots \setminus \xi_k = \mu$. Then (see \cite{Morotti2011}, Corollary~$50$):
\begin{equation*}
\prod_{i=1}^k (-1)^{ht(\xi_i)} = \delta_2(\lambda, \mu).
\end{equation*}

As the product does not depend on the individual $\xi_i$ but only on the end result, we can rewrite \eqref{eq:sum_of_xi} as a sum over $\mu$. This gives the theorem.
\end{proof}

\subsection{Valuations}

If $p$ is a prime number, we let $\nu_p : \mathbb{Q} \rightarrow \mathbb{Z}$ be the $p$-adic (additive) valuation. For any integer $a$, $\nu_p(a)$ is defined as the maximum exponent $r$, such that $p^r \mid a$, with the convention that $\nu_p(0) = \infty$. This extends to $\mathbb{Q}$ by setting $\nu_p(\frac{a}{b}) = \nu_p(a) - \nu_p(b)$. Then $\nu_p$ has the following properties:
\begin{enumerate}[(i)]
\item $\nu_p(ab) = \nu_p(a) + \nu_p(b)$,
\item $\nu_p(-a) = \nu_p(a)$,
\item $\nu_p(a+b) \geq \min\{\nu_p(a), \nu_p(b)\}$, with equality if $\nu_p(a) \neq \nu_p(b)$.
\end{enumerate}

An important special case is the value of $\nu_p(n!)$. This is given by Legendre's formula.
\begin{lemma}
\label{lem:legendre}
Let $p \geq 2$ be a prime number and $n \geq 1$. Then
\begin{equation*}
\nu_p(n!) = \frac{n - \kappa_p(n)}{p-1},
\end{equation*}
where $\kappa_p(n)$ is the sum of the digits of $n$ in base $p$, i.e. if $n = a_0 + a_1p + \ldots + a_rp^r$ with $0 \leq a_i < p$, then $\kappa_p(n) = \sum_{i=0}^r a_i$.
\end{lemma}

We will also need an estimate for $\nu_p(2^\ell-1)$. This is a consequence of the following more general inequality (see \cite{Stewart13}, $(6.5)$).
\begin{lemma}
\label{lem:nu_inequality}
Let $a > b > 0$ be integers, $p > 2$ a prime which does not divide $ab$, and $n \geq 2$. Then
\begin{equation*}
\nu_p(a^n - b^n) \leq \nu_p(a^{p-1} - b^{p-1}) + \nu_p(n).
\end{equation*}
\end{lemma}

\subsection{Valuations of character degrees}

Let $\lambda \vdash n$ be any partition and $p$ a prime number. The value of $\nu_p(F_\lambda)$ was determined by Macdonald (\cite{Macdonald71}). It is a generalization of Lemma~\ref{lem:legendre}, which corresponds to the case $\lambda = (n)$. However, in order to state this result we will need some more terminology.

As we did for $q = 2$, we can define for any $q \geq 1$, the $q$-weight $w_q(\lambda)$ of the partition $\lambda$ as the maximum number of $q$-hooks we can recursively remove from $\lambda$. In particular $w_1(\lambda) = |\lambda| = n$. For any $i \geq 0$ we define
\begin{equation}
\label{eq:p_core_tower}
\alpha_i(\lambda) := w_{p^i}(\lambda) - pw_{p^{i+1}}(\lambda).
\end{equation}

Then $\alpha_i(\lambda) \geq 0$ and $n = \sum_{i\geq 0} \alpha_i(\lambda)p^i$ (see Proposition $4.5$, \cite{MalleOllson}). With this notation we can state the formula for $\nu_p(F_\lambda)$.

\begin{theorem}[Macdonald, \cite{Macdonald71}]
\label{thm:macdonald}
Let $\lambda \vdash n$ be any partition, $p$ a prime number, and $\alpha_i(\lambda)$ as in \eqref{eq:p_core_tower}. Then
\begin{equation*}
\nu_p(H(\lambda)) = \frac{n - \sum_{i \geq 0}\alpha_i(\lambda)}{p-1},
\end{equation*}
and $\nu_p(F_\lambda) = \nu_p(n!) - \nu_p(H(\lambda))$.
\end{theorem}

\subsection{The Newton polygon}

Let $G(x)=\sum_{i=0}^na_ix^i \in \mathbb{Z}[x]$ be a polynomial of degree $n$ with $a_0 \neq 0$. Then we consider the set of points in the plane:
\begin{equation*}
S = \{(0, \nu_p(a_n)), (1, \nu_p(a_{n-1})), \ldots, (n-1, \nu_p(a_1)), (n, \nu_p(a_0))\}.
\end{equation*}
The \textit{Newton polygon} for $G$ with respect to $p$ is the lower convex hull of $S$.

Newton polygons are a very effective tool in establishing the irreducibility of polynomials. One important feature is that the polygon of a product $f(x)g(x)$ is formed from translates of the polygons for $f(x)$ and $g(x)$.

\begin{lemma}[Dumas, \cite{Dumas}]
\label{lem:dumas}
Let $f(x)$ and $g(x)$ be polynomials in $\mathbb{Z}[x]$ with $f(0)g(0) \neq 0$, and let $p$ be a prime. Let $k$ be a non-negative integer such that $p^k$ divides the leading coefficient of $f(x)g(x)$ but $p^{k+1}$ does not. Then the edges of the Newton polygon for $f(x)g(x)$ with respect to $p$ can be formed by constructing a polygonal path beginning at $(0, k)$ and using translates of the
edges in the Newton polygons for $f(x)$ and $g(x)$ with respect to the prime $p$ (using exactly one translate for each edge). Necessarily, the translated edges are translated in such a
way as to form a polygonal path with the slopes of the edges increasing.
\end{lemma}

Our main tool in establishing irreducibility will be the following lemma, due to Filaseta.

\begin{lemma}[\cite{Filaseta95}, Lemma~$2$]
\label{lem:filaseta}
Let $k$ and $\ell$ be integers with $k > \ell \geq 0$. Suppose $G(x)=\sum_{i=0}^na_ix^i \in \mathbb{Z}[x]$ and $p$ is a prime such that $p \centernot\mid a_n, p \mid a_j$ for all $j \in \{0, 1, \ldots, n- \ell -1\}$, and the right-most edge of the Newton polygon for $G(x)$ with respect to $p$ has slope $<1/k$. Then $G(x)$ cannot have a factor with degree in the interval $[\ell+1, k]$.
\end{lemma}

Lemma~\ref{lem:filaseta} played a central role in Filaseta's proof that all but finitely many Bessel polynomials are irreducible (\cite{Filaseta95}). It was also used to extend this result to all Bessel polynomials (\cite{FilasetaTrifonov}). Variations of it were used to establish the irreducibility of generalized Laguerre polynomials in many cases (\cite{Filaseta02}, \cite{Shorey16}).

In order to apply Lemma~\ref{lem:filaseta}, we will need a lower bound for the largest prime factor of a product of consecutive integers. We are going to use the following result, due to Nair and Shorey.

\begin{theorem}[\cite{NairShorey16}]
\label{thm:nairshorey}
Assume that $k \geq 2, n > 100$ and $n, n + 1, \ldots , n + k − 1$ are all composite integers. Then the product $n(n + 1) \ldots (n + k − 1)$ has a prime factor greater than $4.42k$, unless $n = 125, 224, 2400, 4374$ if $k = 2$, and $n = 350$ if $k = 3$.
\end{theorem}

We are also going to need an effective version of Dirichlet's theorem on arithmetic progressions.

\begin{theorem}[\cite{Cullinan12}, Theorem~$1$]
\label{thm:cullinan}
If $x \geq 887$ then the interval $(x, 1.048x]$ contains a prime $p$ with $p \equiv 3 \mod 4$.
\end{theorem}

Finally, we will need a result about the existence of primes in small intervals.

\begin{theorem}[Nagura, \cite{Nagura52}]
\label{thm:nagura}
If $x \geq 25$ is a real number, then there is a prime in the interval $[x, \frac{6x}{5}]$.
\end{theorem}

\subsection{The roots of Hermite polynomials}

Lemma~\ref{lem:bound_real_roots} stated in the Introduction gives an upper bound in terms of the roots of Hermite polynomials. To make this estimate effective we will need an upper bound for the absolute value of the roots of $\He_n(x)$. For the classical Hermite polynomials, Szeg{\H o} proved the following.
\begin{theorem}[Szeg{\H o}, \cite{Szego75}, $(6.2.18)$]
\label{thm:szego}
If $z$ is a root of $H_n(x)$ then $|z| \leq \frac{\sqrt{2}(n-1)}{\sqrt{n+2}}$.
\end{theorem}
From the rescaling $\He_n(x) = 2^{-\frac{n}{2}}H_n(\frac{x}{\sqrt{2}})$ we get
\begin{corollary}
\label{cor:szego}
If $z$ is a root of $\He_n(x)$ then $|z| \leq \frac{2(n-1)}{\sqrt{n+2}}$.
\end{corollary}

We will also rely on the symmetry of $\He_\lambda(x)$. It is known that $\He_\lambda(-x) = (-1)^{|\lambda|}\He_\lambda(x)$ (see Lemma~$3.6$, \cite{BonneuxStevens}). Hence we have the following.
\begin{observation}
\label{obs:symmetric_roots}
If $z$ is a root of $\He_\lambda(x)$ then $-z$ is also a root of $\He_\lambda(x)$.
\end{observation}

\subsection{The coefficients of $\He_\lambda(x)$}

In \cite{BonneuxCoefficients}, the authors determine the coefficients of $\He_\lambda(x)$ in terms of the character $\chi^\lambda$. This result will play a central role in our proofs, so we state it here.
\begin{theorem}[\cite{BonneuxCoefficients}, Theorem~$4.2$]
\label{thm:bonneux}
Let $\lambda \vdash n$. Then
\begin{equation*}
\He_\lambda(x) = \sum_{k=0}^{\lfloor n/2 \rfloor} (-1)^k H(\lambda)\frac{\chi^\lambda(2^k1^{n-2k})}{2^k(n-2k)!k!}x^{n-2k}.
\end{equation*}
\end{theorem}

The subleading coefficient of $\He_\lambda(x)$ can be determined more precisely in terms of the partition $\lambda$.
\begin{theorem}[\cite{BonneuxCoefficients}, Proposition~$4.16$]
\label{thm:subleading_coeff}
Let $\lambda \vdash n$. Then the coefficient of $x^{n-2}$ in $\He_\lambda(x)$ equals $-\frac{1}{2}\sum_{i=1}^{\ell(\lambda)}\lambda_i(\lambda_i-(2i-1))$.
\end{theorem}

Furthermore, in \cite{BonneuxCoefficients}, Theorem~$3.1$, they show that $\He_\lambda(x) = x^{|\bar{\lambda}|}R_\lambda(x)$, where recall that $R_\lambda(x)$ is the remainder polynomial. It can be shown that $|\bar{\lambda}| = \frac{d_\lambda(d_\lambda+1)}{2}$, so this coincides with the decomposition stated in the Introduction. The $2$-core of a partition is always of the form $(m, m-1, \ldots, 1)$ for some $m \geq 0$. Therefore $|\bar{\lambda}| = \binom{m+1}{2}$ for some $m \geq 0$.

\section{An extremal problem for character degrees}
\label{sec:degrees}

In this section we are going to study the following extremal problem: what is the minimum value of $\nu_p(F_\lambda)$ over partitions $\lambda$ of fixed size and fixed $2$-core? The condition on the $2$-core makes the problem non-trivial: otherwise $F_{(n)} = 1$ would show that the minimum is $0$.

For a prime $p$ and non-negative integers $n$ and $m$, we define the function

\begin{equation*}
\ex(n, m, p) = \min\left\{ \nu_p(F_\lambda): \textrm{$\lambda \vdash n$ and $\lambda$ has $2$-core of size $\binom{m+1}{2}$}\right\}.
\end{equation*}

Trivially
\begin{equation}
\label{eq:ex_trivial}
\ex(n, m, p) \geq 0,
\end{equation}
and in fact this is best possible for $m = 0$ and $n$ even, as the partition $\lambda = (n)$ shows.

However, it turns out that when $n$ is small and $m \geq 1$, inequality \eqref{eq:ex_trivial} can be slightly improved. For this, we are going to restrict ourselves to the range $m \geq 1$ and $p \geq 2m - 1$. For $m$ and $p$ in this range, we define the numbers
\begin{equation*}
N_k := kp - k(2(m-k)+1) + \binom{m+1}{2}, \quad\quad 0 \leq k \leq \left\lfloor \frac{m+1}{2} \right\rfloor.
\end{equation*}
Then $N_0 = \binom{m+1}{2}$ and $N_{\left\lfloor \frac{m+1}{2} \right\rfloor} = \left\lfloor \frac{m+1}{2} \right\rfloor p$. Furthermore, for $k < \left\lfloor \frac{m+1}{2} \right\rfloor$, the condition $N_k \leq N_{k+1}$ is equivalent to $p \geq 2(m-2k) - 1$, which holds by assumption. Therefore $N_0 \leq N_1 \leq \ldots \leq N_{\left\lfloor \frac{m+1}{2} \right\rfloor}$.

We extend this sequence by defining $N_{\left\lfloor \frac{m+1}{2} \right\rfloor + 1} = (\left\lfloor \frac{m+1}{2} \right\rfloor + 1) p$. As $p \geq 2m - 1$ and $p \geq 2$, this number is at most $p^2$.

We are going to prove the following result.

\begin{lemma}
\label{lem:ex_inequality}
Suppose $m \geq 1$ and $p \geq 2m - 1$ is a prime. If $0 \leq k \leq \left\lfloor \frac{m+1}{2} \right\rfloor + 1$ and $n < N_k$ then
\begin{equation*}
\ex(n, m, p) \geq \nu_p(n!) - (k-1).
\end{equation*}
\end{lemma}

For $k \leq \left\lfloor \frac{m+1}{2} \right\rfloor$, the value of $N_k$ is tight for the inequality in Lemma~\ref{lem:ex_inequality}, as can be seen by taking $\lambda$ the partition $(m ,m-1, \ldots, 1)$ and adding $p-2(m-k)-1$ to the first $k$ rows. In other words, $\lambda = (m+p-(2(m-k)+1), m-1 + p-(2(m-k)+1), \ldots, m-k+1 + p-(2(m-k)+1), m-k, \ldots, 1)$. Then $|\lambda| = N_k$ and $\lambda$ has degree sequence $n_\lambda = (1, 3, \ldots, 2(m-k)-1, p, p+2, \ldots, p+2(k-1))$. As $p \geq 2m - 1$, $\Delta(n_\lambda)$ is not divisible by $p$, while $n_{\lambda, 1}!n_{\lambda, 2}!\ldots n_{\lambda,m}!$ is exactly divisible by $p^k$. Hence $\nu_p(F_\lambda) = \nu_p(N_k!) - k$.

To prove Lemma~\ref{lem:ex_inequality}, we are first going to show the following.
\begin{lemma}
\label{lem:wp_bound}
Let $m \geq 1$ and $\lambda \vdash n$ be a partition with $2$-core of size $\binom{m+1}{2}$. If $p$ is a prime and $w_p(\lambda) = t \leq \left\lfloor \frac{m+1}{2} \right\rfloor$ then $n \geq N_t$.
\end{lemma}
\begin{proof}
Because $\lambda$ has $2$-core of size $\binom{m+1}{2}$, the first $m$ rows of the Ferrers diagram of $\lambda$ contain the diagram of the partition $\mu := (m, m-1, \ldots, 1)$. Then $D_\mu \subseteq D_\lambda$. We consider this as an embedding, so that elements of $D_\mu$ are also elements of $D_\lambda$.

Because $w_p(\lambda) = t$, we can remove $t$ $p$-hooks from $\lambda$. Each such hook intersects $D_\mu$ in a (possibly empty) set of squares. By removing the $t$ hooks in a valid arbitrary, but fixed way, we obtain a sequence of partitions $\mu = \mu_0 \geq \mu_1 \geq \ldots \geq \mu_t$. Again we consider them embedded in $\lambda$, so $D_{\mu_t} \subseteq D_{\mu_{t-1}} \subseteq \ldots \subseteq D_\mu$. From the hook definition, $D_{\mu_{k-1}} \setminus D_{\mu_k}$ is a set of border squares in $\mu_{k-1}$.

Now let $\rho_0 := \mu$ and define $\rho_k$ as the partition obtained from $\rho_{k-1}$ by removing all border squares, for $1 \leq k \leq t$. Then $\rho_k$ is a $2$-core corresponding to $(m-2k, m-2k-1, \ldots, 1)$. As $t \leq \left\lfloor \frac{m+1}{2} \right\rfloor$, $\rho_k$ is not empty for $k < t$.

Here is an example for $m=5$.
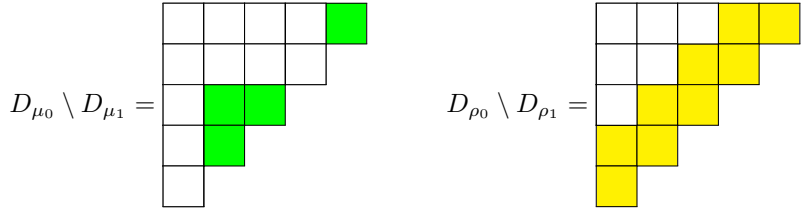
\begin{equation*}
\ytableausetup{centertableaux, nosmalltableaux}
D_{\mu_0}\setminus D_{\mu_1} = \ydiagram[*(white)]{4,4,1,1,1} * [*(green)]{5,4,3,2,1}
\quad\quad\quad
D_{\rho_0}\setminus D_{\rho_1} = \ydiagram[*(white)]{3,2,1,0,0} * [*(yellow)]{5,4,3,2,1}
\end{equation*}
In the left diagram, the green squares represent the intersection of the first hook with $\mu$. They are removed to obtain $\mu_1$. In the right diagram, the yellow squares are removed to form $\rho_1$. Note that the green squares need not be contiguous along the border, because part of the hook can be in $\lambda / \mu$.

We claim that
\begin{equation}
\label{eq:ferrers_inclusion}
D_\mu \setminus D_{\mu_k} \subseteq D_\mu \setminus D_{\rho_k} \textrm{ for $0 \leq k \leq t$.}
\end{equation}

We prove this claim by induction on $k$.

If $k = 0$, then $\mu_0 = \rho_0 = \mu$, and there is nothing to prove.

So assume $k \geq 1$ and \eqref{eq:ferrers_inclusion} holds for $k-1$. Let $(a, b) \in D_\mu \setminus D_{\mu_k}$. If $(a, b) \in D_\mu \setminus D_{\mu_{k-1}}$, then by induction and the fact that $\rho_{k-1} \geq \rho_k$, $(a, b) \in D_\mu \setminus D_{\rho_{k-1}} \subseteq D_\mu \setminus D_{\rho_k}$. So we may assume that $(a, b) \in D_{\mu_{k-1}}$. This means $(a, b)$ is removed from $\mu_{k-1}$ to form $\mu_k$, so it is a border square. Any neighbor $(a+1, b), (a, b+1)$ or $(a+1, b+1)$ that is missing in $D_{\mu_{k-1}}$ is also missing in $D_{\rho_{k-1}}$ by induction. So either $(a, b) \notin D_{\rho_{k-1}}$, in which case we are done, or $(a, b)$ is a border square in $\rho_{k-1}$. In the latter case, $(a, b)$ is removed to form $\rho_k$. Hence $(a, b) \in D_\mu\setminus D_{\rho_k}$. This proves \eqref{eq:ferrers_inclusion}.

By taking $k = t$ in \eqref{eq:ferrers_inclusion} we deduce that $|\mu_t| \geq |\rho_t|$.

However, $|\rho_t| = \binom{m-2t+1}{2}$. Hence
\begin{equation*}
n \geq pt + |\mu_t| \geq pt + \binom{m-2t+1}{2} = pt - t(2(m - t) + 1) + \binom{m+1}{2} = N_t.
\end{equation*}
\end{proof}

\begin{proof}[Proof of Lemma~\ref{lem:ex_inequality}]
First assume $k \leq \left\lfloor \frac{m+1}{2} \right\rfloor$. Let $\lambda \vdash n$ be any partition with $2$-core of size $\binom{m+1}{2}$.

Let $\alpha_i(\lambda)$ be defined as in \eqref{eq:p_core_tower}. As $n < N_k < p^2$, $w_{p^i}(\lambda) = 0$ for $i \geq 2$. Hence $\alpha_i(\lambda) = 0$ for all $i \geq 2$. Furthermore,
\begin{equation*}
\alpha_1(\lambda) = w_p(\lambda) - pw_{p^2}(\lambda) = w_p(\lambda).
\end{equation*}

From Theorem~\ref{thm:macdonald},
\begin{align*}
\nu_p(H(\lambda)) &= \frac{n-\alpha_0(\lambda)-\alpha_1(\lambda)}{p-1}\\
&= \frac{\alpha_1(\lambda)p - \alpha_1(\lambda)}{p-1}, \textrm{ as $n = \alpha_0(\lambda) + \alpha_1(\lambda)p$,}\\
&= \alpha_1(\lambda) = w_p(\lambda).
\end{align*}

From Lemma~\ref{lem:wp_bound}, $w_p(\lambda) \leq k - 1$. Then $\nu_p(H(\lambda)) \leq k - 1$. Hence
\begin{equation*}
\nu_p(F_\lambda) = \nu_p(n!) - \nu_p(H(\lambda)) \geq \nu_p(n!) - (k-1),
\end{equation*}
proving the lemma for $k \leq \left\lfloor \frac{m+1}{2} \right\rfloor$.

Now assume $k = \left\lfloor \frac{m+1}{2} \right\rfloor + 1$. Then $n < N_k = kp \leq p^2$ and so $\nu_p(n!) \leq k - 1$. Then $\nu_p(n!) - (k-1) \leq 0$ and the claim of the lemma is trivially true in this case.
\end{proof}

\section{An upper bound for the slope}
\label{sec:slope}

Let $\lambda \vdash n$ be any partition. In order to show that $R_\lambda(x)$ is irreducible using Lemma~\ref{lem:filaseta}, we need to estimate the slope of the right-most edge of the Newton polygon for $R_\lambda(x)$.

Let $s:=|\bar{\lambda}|$ and $w:= w_2(\lambda)$. From Theorem~\ref{thm:bonneux} and the fact that $\He_\lambda(x) = x^sR_\lambda(x)$ we get
\begin{equation*}
R_\lambda(x) = \sum_{k=0}^w (-1)^kH(\lambda)\frac{\chi^\lambda(2^k1^{n-2k})}{2^k(n-2k)!k!}x^{n-s-2k}.
\end{equation*}
We now define $r^\lambda_{2k} := \frac{\chi^\lambda(2^k1^{n-2k})}{(n-2k)!k!}$ for $0 \leq k \leq w$. Then $R_\lambda(x) = \sum_{k=0}^w (-1)^k\frac{H(\lambda)}{2^k}r^\lambda_{2k}x^{n-s-2k}$. If $p > 2$ is a prime number, then using properties (i) and (ii) of $\nu_p$, the slope of the right-most edge of the Newton polygon for $R_\lambda(x)$ with respect to $p$ is given by
\begin{equation*}
\max_{0 \leq k < w}\left\{\frac{\nu_p(r^\lambda_{2w}) - \nu_p(r^\lambda_{2k})}{n-s-2k}\right\}.
\end{equation*}
We will estimate this in several steps.

\begin{lemma}
\label{lem:r_2w_estimate}
Let $\lambda \vdash n$ be a partition with $2$-quotient $(\pi_0, \pi_1)$ and $2$-weight $w$. If $p > 2$ is a prime number that does not divide $\Delta(n_\lambda)$, then
\begin{equation*}
\nu_p(r^\lambda_{2w})=-\sum_{i=1}^{\ell(\pi_0)}\nu_p(n_{\pi_0, i}!)-\sum_{i=1}^{\ell(\pi_1)}\nu_p(n_{\pi_1, i}!) -\nu_p(H(\bar{\lambda})).
\end{equation*}
\end{lemma}
\begin{proof}
Let $r_0 := \ell(\pi_0)$ and $r_1 := \ell(\pi_1)$.

Using Theorem~\ref{thm:value2weight}, expression \eqref{eq:hnfactorial} for $F_{\pi_0}$ and $F_{\pi_1}$, and the fact that $w = |\pi_0|+|\pi_1|$, we obtain
\begin{equation*}
\chi^\lambda(2^w1^{n-2w}) =\delta_2(\lambda) w! \frac{\Delta(n_{\pi_0})}{n_{\pi_0,1}!n_{\pi_0, 2}!\ldots n_{\pi_0, r_0}!} \frac{\Delta(n_{\pi_1})}{n_{\pi_1,1}!n_{\pi_1, 2}!\ldots n_{\pi_1, r_1}!} F_{\bar{\lambda}}.
\end{equation*}
Dividing by $(n-2w)!w!$ and using $F_{\bar{\lambda}} = (n-2w)!/H(\bar{\lambda})$, we get
\begin{equation*}
r^\lambda_{2w} =\delta_2(\lambda) \frac{\Delta(n_{\pi_0})}{n_{\pi_0,1}!n_{\pi_0, 2}!\ldots n_{\pi_0, r_0}!} \frac{\Delta(n_{\pi_1})}{n_{\pi_1,1}!n_{\pi_1, 2}!\ldots n_{\pi_1, r_1}!} \frac{1}{H(\bar{\lambda})}.
\end{equation*}
Consequently
\begin{equation}
\label{eq:nu_of_2w}
\nu_p(r^\lambda_{2w}) = \nu_p(\Delta(n_{\pi_0})) + \nu_p(\Delta(n_{\pi_1})) -\sum_{i=1}^{r_0}\nu_p(n_{\pi_0, i}!)-\sum_{i=1}^{r_1}\nu_p(n_{\pi_1, i}!) -\nu_p(H(\bar{\lambda}))
\end{equation}

Let $t \in \{0, 1\}$. We claim that $\nu_p(\Delta(n_{\pi_t})) = 0$. To see this, note that $\nu_p(\Delta(n_{\pi_t})) = \sum_{i < j}\nu_p(n_{\pi_t, j} - n_{\pi_t, i})$. Fix a pair $1 \leq i < j \leq r_t$. By definition, there exists $i' \neq j'$ such that $n_{\pi_t, i} = \frac{n_{\lambda, i'} - t}{2}$ and $n_{\pi_t, j} = \frac{n_{\lambda, j'} - t}{2}$. But $p$ does not divide $n_{\lambda, j'} - n_{\lambda, i'}$, so $\nu_p(n_{\pi_t, j} - n_{\pi_t, i}) = \nu_p(n_{\lambda, j'} - n_{\lambda, i'}) = 0$. Hence $\nu_p(\Delta(n_{\pi_t})) = 0$ for $t \in \{0, 1\}$.

Applying this to \eqref{eq:nu_of_2w} finishes the proof.
\end{proof}

If $\mu \leq \lambda$, define $A_{\lambda/\mu}$ as the matrix $\left(\frac{1}{(n_{\lambda, i} - n_{\mu, j})!}\right)_{i,j=1}^{\ell(\lambda)}$, with the convention that $1/m!$ is $0$ if $m$ is negative. This is well-defined, as we require that $\mu$ and $\lambda$ have the same length when $\mu \leq \lambda$.

\begin{lemma}
\label{lem:r_2k_estimate}
Let $\lambda \vdash n$ be a partition with $2$-quotient $(\pi_0, \pi_1)$, $2$-weight $w$ and $2$-core of size $\binom{m+1}{2}$.
If $p > 2$ is a prime number, and $k \leq w$, then
\begin{equation*}
\nu_p(r^\lambda_{2k}) \geq -\nu_p((n-2k)!) + \min_{\mu \leq_k \lambda} \{ \nu_p(\det[A_{\pi_0/\rho_0}]) + \nu_p(\det[A_{\pi_1/\rho_1}])\} + \ex(n-2k, m, p),
\end{equation*}
where $(\rho_0, \rho_1)$ is the $2$-quotient of $\mu$.
\end{lemma}
\begin{proof}
By Theorem~\ref{thm:valuelessthan2weight} and property (iii) of $\nu_p$,
\begin{equation}
\label{eq:thm8_consequence}
\nu_p(r^\lambda_{2k}) \geq -\nu_p((n-2k)!) - \nu_p(k!) + \min_{\mu \leq_k \lambda} \{ \nu_p(F_{2,\lambda/\mu}) + \nu_p(F_\mu) \}.
\end{equation}
Fix $\mu \leq_k \lambda$ and assume it has $2$-quotient $(\rho_0, \rho_1)$. As we require that $\ell(\mu) = \ell(\lambda)$, it follows that $\ell(\rho_0) = \ell(\pi_0)$ and $\ell(\rho_1) = \ell(\pi_1)$. Removing a $2$-hook is the same as removing a corner square from either $\pi_0$ or $\pi_1$. Hence
\begin{equation}
\label{eq:f2lambdamu}
F_{2,\lambda/\mu} = \binom{k}{|\pi_0|-|\rho_0|}F_{\pi_0/\rho_0}F_{\pi_1/\rho_1},
\end{equation}
where $F_{\pi_t/\rho_t}$ represents the number of ways we can obtain $\rho_t$ from $\pi_t$ by removing corner squares, for $t \in \{0, 1\}$. This is given by Aitken's formula (\cite{Aitken43}, see also \cite{StanleyVol2}, Corollary $7.16.3$):
\begin{equation}
\label{eq:aitken}
F_{\pi_t/\rho_t} = (|\pi_t| - |\rho_t|)!\det\left[\frac{1}{(n_{\pi_t, i} - n_{\rho_t, j})!}\right]_{i,j=1}^{\ell(\pi_t)}, t \in \{0, 1\}.
\end{equation}
Replacing \eqref{eq:aitken} into \eqref{eq:f2lambdamu} gives
\begin{equation*}
F_{2,\lambda/\mu} = k! \det[A_{\pi_0/\rho_0}] \det[A_{\pi_1/\rho_1}].
\end{equation*}

Now consider the term $\nu_p(F_\mu)$ in \eqref{eq:thm8_consequence}. Removing $2$-hooks does not change the $2$-core, so the partition $\mu$ has the same $2$-core as $\lambda$. In particular, $|\bar{\mu}| = \binom{m+1}{2}$. Then by definition, $\nu_p(F_\mu) \geq \ex(n - 2k, m, p)$. Therefore
\begin{align*}
\nu_p(r^\lambda_{2k}) &\geq -\nu_p((n-2k)!) - \nu_p(k!) + \min_{\mu \leq_k \lambda} \{ \nu_p(k!) + \nu_p(\det[A_{\pi_0/\rho_0}]) + \nu_p(\det[A_{\pi_1/\rho_1}]) + \ex(n - 2k, m, p)\}\\
&= -\nu_p((n-2k)!) + \min_{\mu \leq_k \lambda} \{ \nu_p(\det[A_{\pi_0/\rho_0}]) + \nu_p(\det[A_{\pi_1/\rho_1}])\} + \ex(n - 2k, m, p).
\end{align*}
\end{proof}

We now put together the previous two results to obtain the following.
\begin{lemma}
\label{lem:diffbound}
Let $\lambda \vdash n$ be a partition with $2$-weight $w$ and $2$-core of size $\binom{m+1}{2}$. If $p > 2$ is a prime number that does not divide $\Delta(n_\lambda)$, and $k \leq w$, then
\begin{equation*}
\nu_p(r^\lambda_{2w}) - \nu_p(r^\lambda_{2k}) \leq \nu_p((n-2k)!) - \ex(n-2k,m,p) - \nu_p(H(\bar{\lambda})).
\end{equation*}
\end{lemma}
\begin{proof}
Assume $\lambda$ has $2$-quotient $(\pi_0, \pi_1)$. Let $r_0 := \ell(\pi_0)$ and $r_1 := \ell(\pi_1)$. From Lemmas~\ref{lem:r_2w_estimate} and~\ref{lem:r_2k_estimate}, it follows that
\begin{align}
\nu_p(r^\lambda_{2w}) - \nu_p(r^\lambda_{2k}) \leq &-\sum_{i=1}^{r_0}\nu_p(n_{\pi_0, i}!)-\sum_{i=1}^{r_1}\nu_p(n_{\pi_1, i}!) -\nu_p(H(\bar{\lambda}))\label{eq:nucombined} \\
&+\nu_p((n-2k)!) + \max_{\mu \leq_k \lambda} \{ -\nu_p(\det[A_{\pi_0/\rho_0}]) - \nu_p(\det[A_{\pi_1/\rho_1}])\} - \ex(n-2k,m,p)\nonumber,
\end{align}
where $(\rho_0, \rho_1)$ is the $2$-quotient of $\mu$.

Fix $\mu \leq_k \lambda$ and $t \in \{0, 1\}$. We claim that
\begin{equation}
\label{eq:max_mu_sum}
-\sum_{i=1}^{r_t}\nu_p(n_{\pi_t, i}!) - \nu_p(\det[A_{\pi_t/\rho_t}]) \leq -\sum_{i=1}^{r_t} \nu_p(n_{\rho_t, i}!) \leq 0.
\end{equation}
Indeed, by properties (i) and (iii) of $\nu_p$, we have that
\begin{equation*}
\nu_p(\det[A_{\pi_t/\rho_t}]) \geq \min_{\sigma :[r_t] \rightarrow [r_t]} \sum_{i=1}^{r_t} \nu_p\left(\frac{1}{(n_{\pi_t, i} - n_{\rho_t, \sigma(i)})!}\right),
\end{equation*}
where the minimum is taken over all permutations $\sigma$.

Recall our convention that $1/m! = 0$ for elements of the matrix $A_{\lambda/\mu}$ when $m$ is negative. In this case we also make the convention that $\nu_p(m!) := -\infty$ for negative $m$. Using this convention we recover the property $-\nu_p(\frac{1}{m!}) = \nu_p(m!)$. Consequently
\begin{equation*}
-\sum_{i=1}^{r_t}\nu_p(n_{\pi_t, i}!) - \nu_p(\det[A_{\pi_t/\rho_t}]) \leq \max_{\sigma :[r_t] \rightarrow [r_t]} \sum_{i=1}^{r_t} (\nu_p((n_{\pi_t, i} - n_{\rho_t, \sigma(i)})!) - \nu_p(n_{\pi_t, i}!)).
\end{equation*}
Fix $\sigma : [r_t] \rightarrow [r_t]$ for the moment. We argue that
\begin{equation}
\label{eq:sigma_sum}
\sum_{i=1}^{r_t} (\nu_p((n_{\pi_t, i} - n_{\rho_t, \sigma(i)})!) - \nu_p(n_{\pi_t, i}!)) \leq -\sum_{i=1}^{r_t} \nu_p(n_{\rho_t, \sigma(i)}!).
\end{equation}
If $n_{\rho_t, \sigma(i)} > n_{\pi_t, i}$ for some $i$, then the left-hand side sum equals $-\infty$. If $n_{\rho_t, \sigma(i)} \leq n_{\pi_t, i}$ for all $i$, then as $\binom{n_{\pi_t, i}}{n_{\pi_t, i} - n_{\rho_t, \sigma(i)}}$ is a positive integer, the term $\nu_p((n_{\pi_t, i} - n_{\rho_t, \sigma(i)})!) - \nu_p(n_{\pi_t, i}!)$ is at most $-\nu_p(n_{\rho_t, \sigma(i)}!)$. Thus \eqref{eq:sigma_sum} is true.

Now the right-hand side sum in \eqref{eq:sigma_sum} is just $-\sum_{i=1}^{r_t} \nu_p(n_{\rho_t, i}!)$, so taking the maximum over all $\sigma$ does not change it. This shows the middle inequality in \eqref{eq:max_mu_sum}. As all the numbers $n_{\rho_t, i}$ are non-negative integers, the entire sum is at most $0$.

Using \eqref{eq:max_mu_sum} in \eqref{eq:nucombined} for the $\mu$ which attains the maximum proves the lemma.
\end{proof}

We are now ready to determine an upper bound for the slope of the right-most edge.
\begin{proof}[Proof of Theorem~\ref{thm:slope}]
Let $w$ be the $2$-weight of $\lambda$. From Lemma~\ref{lem:diffbound} we know that the slope is at most
\begin{equation*}
\max_{0 \leq k \leq w}\frac{\nu_p((n-2k)!) - \ex(n-2k, m, p) - \nu_p(H(\bar{\lambda}))}{n-s-2k}.
\end{equation*}

$H(\bar{\lambda})$ is a positive integer (\cite{JamesReprTheory}, Theorem~$2.3.21$), so $\nu_p(H(\bar{\lambda})) \geq 0$. Then the slope is at most
\begin{equation}
\label{eq:first_slope_bound}
\max_{0 \leq k \leq w}\frac{\nu_p((n-2k)!) - \ex(n-2k, m, p)}{n-s-2k}.
\end{equation}

As $\ex(n-2k, m, p) \geq 0$ we get that

\begin{equation}
\label{eq:second_slope_bound}
\max_{0 \leq k \leq w} \frac{\nu_p((n-2k)!) - \ex(n-2k, m, p)}{n-s-2k} \leq \max_{0 \leq k \leq w} \frac{\nu_p((n-2k)!)}{n-s-2k}.
\end{equation}

We will use the simpler form \eqref{eq:second_slope_bound} to bound the slope for most of the proof, returning to \eqref{eq:first_slope_bound} only at the end to handle the most difficult case.

Fix $k \leq w$. Lemma~\ref{lem:legendre} implies that $\nu_p(t!) < \frac{t}{p-1}$ for any $t \geq 1$. Hence if $p$ does not divide any of the numbers $n-2k, n-2k-1, \ldots, n-2k-s+1$, then
\begin{equation*}
\frac{\nu_p((n-2k)!)}{n-s-2k} = \frac{\nu_p((n-s-2k)!)}{n-s-2k}  < \frac{1}{p-1}.
\end{equation*}

In the special case $s=0$ and $n < p^2$ we can strengthen this to
\begin{equation*}
\frac{\nu_p((n-2k)!)}{n-2k} = \frac{a}{n-2k} \leq \frac{a}{ap} = \frac{1}{p},
\end{equation*}
where $a:=\lfloor\frac{n-2k}{p}\rfloor$. This proves (i).

From now on we will assume that $n-2k-i = ap^j$ for some $0 \leq i < s, j \geq 1$ and $(a, p) = 1$. In particular, $s > 0$ and so $m \geq 1$. We choose $i$ minimal, so $p$ does not divide any of the numbers $n-2k, n-2k-1, \ldots, n-2k-i+1$. We will split the analysis in several cases, depending on the value of $m$ and $j$.

\begin{case}
$m=1$.
\end{case}

Then $s = 1$ and $n-2k = ap^j$. Hence by Lemma~\ref{lem:legendre},
\begin{equation*}
\nu_p((n-2k)!) = j + \nu_p((n-1-2k)!) = j + \frac{n-1-2k - \kappa_p(n-1-2k)}{p-1}.
\end{equation*}
But $n-1-2k = ap^j - 1 = (a-1)p^j + p^j - 1$. Then the first $j$ digits of $n-1-2k$ in base $p$ are equal to $p-1$. Hence $\kappa_p(n-1-2k) \geq j(p-1)$. Therefore
\begin{equation*}
\nu_p((n-2k)!) \leq j + \frac{n-1-2k - j(p-1)}{p-1} = \frac{n-1-2k}{p-1}.
\end{equation*}
So
\begin{equation*}
\frac{\nu_p((n-2k)!)}{n-1-2k} \leq \frac{1}{p-1} = \frac{1}{p-(2m-1)}.
\end{equation*}
This shows (ii) in this case.

\begin{case}
$m > 1$, and either $j \geq 2$ or $a \geq \frac{m}{4} + 1$.
\end{case}

First we claim that
\begin{equation}
\label{eq:case_m_2}
ap^j \geq \frac{s(p-1)}{2(m-1)}.
\end{equation}

Indeed, as $m \geq 2$ we have
\begin{equation*}
\frac{s}{2(m-1)} = \frac{1}{2(m-1)}\binom{m+1}{2} = \frac{m(m+1)}{4(m-1)} \leq \frac{m}{4} + 1.
\end{equation*}
As $p \geq 2m + 1 > \frac{m}{4} + 1$, and either $j \geq 2$ or $a \geq \frac{m}{4} + 1$, we see that $ap^{j-1} \geq \frac{m}{4} + 1 \geq \frac{s}{2(m-1)}$. Multiplying this with $p$ shows that \eqref{eq:case_m_2} holds.

Now by our choice of $i$, $\nu_p((n-2k)!) = \nu_p((ap^j)!) < \frac{ap^j}{p-1}$. Therefore
\begin{equation*}
\frac{\nu_p((n-2k)!)}{n-s-2k} < \frac{ap^j}{(p-1)(n-s-2k)} \leq \frac{ap^j}{(p-1)(ap^j-s)}.
\end{equation*}

We can bound the right-hand side in the following way
\begin{align*}
\frac{ap^j}{(p-1)(ap^j-s)} &= \frac{1}{p-1} + \frac{s}{(p-1)(ap^j - s)}\\
&\leByRef{eq:case_m_2} \frac{1}{p-1} + \frac{2(m-1)}{(p-1)(p-(2m-1))}\\
&= \frac{1}{p-(2m-1)}.
\end{align*}

This proves (ii) in this case.

\begin{case}
$m > 1, j = 1$ and $a \leq \left\lceil \frac{m}{4} \right\rceil$.
\end{case}

This is the most difficult case of the proof. Here is where we are going to use \eqref{eq:first_slope_bound}.

Note that $n-2k \geq ap > n-s-2k$, so $n-2k < ap + s$. We claim that
\begin{equation}
\label{eq:bound_by_Nt}
ap + s \leq \left(\left\lfloor \frac{m+1}{2}\right\rfloor + 1\right)p.
\end{equation}
This is equivalent to
\begin{equation*}
p\left(\left\lfloor \frac{m+1}{2}\right\rfloor + 1 - a\right) \geq \frac{m(m+1)}{2}.
\end{equation*}
Using $a \leq \left\lceil \frac{m}{4} \right\rceil \leq \frac{m+3}{4}$ and $\left\lfloor \frac{m+1}{2} \right\rfloor \geq \frac{m}{2}$, the left-hand side is at least
\begin{equation*}
p\left(\frac{m}{2} + 1 - \frac{m+3}{4}\right) = \frac{p(m+1)}{4} > \frac{m(m+1)}{2},
\end{equation*}
where the last inequality follows from the fact that $p > 2m$. This proves \eqref{eq:bound_by_Nt}.

Consequently there exists $0 \leq t \leq \left\lfloor \frac{m+1}{2}\right\rfloor$ such that $N_t \leq n - 2k < N_{t+1}$, where $N_t$ is defined as in Section~\ref{sec:degrees}. Then $n-s-2k \geq N_t - s$. Also from Lemma~\ref{lem:ex_inequality}, $\ex(n-2k, m, p) \geq \nu_p((n-2k)!)-t$. Using these inequalities in \eqref{eq:first_slope_bound} we obtain
\begin{equation*}
\frac{\nu_p((n-2k)!) - \ex(n-2k, m, p)}{n-s-2k} \leq \frac{t}{n-s-2k} \leq \frac{t}{N_t - s}.
\end{equation*}
If $t = 0$, the middle ratio is $0$, and there is nothing to show. So we may assume $t \geq 1$. By definition of $N_t$ we have
\begin{equation*}
\frac{t}{N_t - s} =\frac{t}{tp - t(2(m-t)+1)} =\frac{1}{p-2(m-t)-1} \leq \frac{1}{p-(2m-1)}.
\end{equation*}

This proves (ii) in this case. Together with the other two cases this proves (ii) in the theorem.

\medskip

We now show the upper bound is essentially sharp for $m \geq 0$ and $p > \max\{2, 2m-1\}$. Again, we consider two cases, depending on the value of $m$.

\setcounter{case}{0}
\renewcommand{\thecase}{\Alph{case}}
\begin{case}
$m \geq 1$.
\end{case}

Choose any $n$ such that $p \mid n + m$. We consider the partition $\lambda:=(m+2n, m-1, \ldots, 2, 1)$. The Ferrers diagram for $\lambda$ when $m = 4$ is displayed in the following picture.
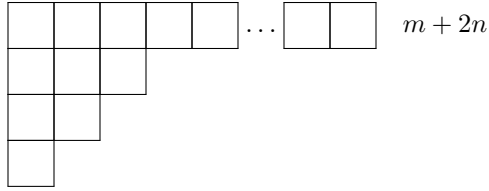
\begin{equation*}
\ytableausetup{centertableaux, mathmode, boxsize=1.7em}
\begin{ytableau}
\none & & & & & & \none[\dots] & & &  \none & \none[m+2n]\\
\none & & & \\
\none & & \\
\none & 
\end{ytableau}
\end{equation*}

Then $\lambda$ has $2$-core $(m, m-1, \ldots, 1)$ of size $s := \binom{m+1}{2}$. Furthermore, $n_\lambda = (1, 3, \ldots, 2m-3, 2m-1 + 2n)$. Hence
\begin{align*}
\Delta(n_\lambda) &= (2m+2n-2)(2m+2n-4)\ldots(2n+2)\Delta(1, 3, \ldots, 2m-3)\\
&= 2^{m-1}(n+1)(n+2)\ldots(n+m-1)\Delta(1, 3, \ldots, 2m-3).
\end{align*}
As $p$ divides $n + m$, and $p > m$, $p$ does not divide the product $(n+1)(n+2)\ldots(n+m-1)$. Also, $\Delta(1, 3, \ldots, 2m-3)$ is a product of numbers less than $2m-3 < p$, so $p$ does not divide it either. Therefore $p$ does not divide $\Delta(n_\lambda)$.

For $k \leq n$ let $\lambda_k := (m+2k, m-1, \ldots, 1)$. Because we can only remove $2$-hooks from the first row, the Murnaghan-Nakayama rule implies that $\chi^\lambda(2^k1^{2n+s-2k}) = F_{\lambda_{n-k}}$ for any $k \leq n$. Therefore
\begin{equation*}
\He_\lambda(x) = \sum_{k=0}^n(-1)^k\frac{H(\lambda)}{2^k(2n+s-2k)!k!}F_{\lambda_{n-k}}x^{2n+s-2k}.
\end{equation*}
It follows that
\begin{equation*}
R_\lambda(x) = \sum_{i=0}^n(-1)^{n-i}\frac{H(\lambda)}{2^{n-i}(s+2i)!(n-i)!}F_{\lambda_i}x^{2i}.
\end{equation*}
Let $j := \frac{p-(2m-1)}{2}$, which is an integer for all $m \geq 1$. Then $j \geq 1$, as $p > 2m-1$. By looking at the coefficients of $x^0$ and $x^{2j}$, the slope of the right-most edge of the Newton polygon is at least
\begin{equation*}
\frac{1}{2j}\left(\nu_p(\frac{H(\lambda)}{s!n!}F_{\lambda_0}) - \nu_p(\frac{H(\lambda)}{(s+2j)!(n-j)!}F_{\lambda_j})\right).
\end{equation*}
We can expand this as
\begin{equation*}
\frac{1}{2j}\left(\nu_p((s+2j)!)+\nu_p((n-j)!)-\nu_p(s!)-\nu_p(n!)+\nu_p(F_{\lambda_0})-\nu_p(F_{\lambda_j})\right).
\end{equation*}
Note that $n -j \geq n + m - p$. Because of this and the fact that $p$ divides $n + m$, $p$ does not divide any of the numbers $n-j+1, n-j+2, \ldots, n$. Hence $\nu_p((n-j)!)=\nu_p(n!)$.

Furthermore for any $0 \leq k \leq n$,
\begin{equation*}
F_{\lambda_k} = \frac{(s+2k)!}{1!3!\ldots (2m-3)!(2m-1+2k)!}\Delta(1, 3, \ldots, 2m-3, 2m-1+2k).
\end{equation*}

Taking $k=0$ we see that $\nu_p(F_{\lambda_0}) - \nu_p(s!) = 0$, as $p > 2m-1$.

Taking $k = j$ we see that $\nu_p((s+2j)!) - \nu_p(F_{\lambda_j}) = \nu_p((2m-1+2j)!) = \nu_p(p!) = 1$.

Putting all of this together we obtain that the slope is at least $\frac{1}{2j} = \frac{1}{p-(2m-1)}$. This shows that the upper bound is tight for $m \geq 1$.

\begin{case}
$m = 0$.
\end{case}
Let $p \geq 3$ arbitrary and choose any $n$ such that $p^2 \mid n$. We consider the partition $\lambda := (2 + 2n, 2, \ldots, 2)$ of length $p$. The Ferrers diagram of $\lambda$ is displayed in the following picture.
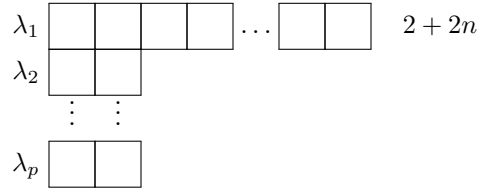
\begin{equation*}
\ytableausetup{centertableaux, mathmode, boxsize=1.7em}
\begin{ytableau}
\none[\lambda_1] & & & & & \none[\dots] & & & \none & \none[2+2n]\\
\none[\lambda_2] & & \\
\none & \none[\vdots] & \none[\vdots]\\
\none[\lambda_p] & & 
\end{ytableau}
\end{equation*}

Then $\lambda$ has $2$-weight $w = n + p$ and empty $2$-core. Furthermore, the degree vector is $n_\lambda = (2, 3, \ldots, p, 2n + p + 1)$. Hence
\begin{equation*}
\Delta(n_\lambda) = (2n + p - 1)(2n + p - 2)\ldots(2n + 1)\Delta(2, 3, \ldots, p).
\end{equation*}
As $p$ divides $n$, $p$ does not divide the product $(2n + p -1)(2n + p -2)\ldots(2n+1)$. Furthermore, $\Delta(2, 3, \ldots, p)$ is a product of numbers less than $p$, hence it is not divisible by $p$. Therefore $p$ does not divide $\Delta(n_\lambda)$.

We compute $\nu_p(\chi^\lambda(2^w))$. Let $(\pi_0, \pi_1)$ be the $2$-quotient of $\lambda$. Then $\pi_0$ is of length $\frac{p+1}{2}$ and has degree sequence $n_{\pi_0} = (1, 2, \ldots, \frac{p-1}{2}, n+\frac{p+1}{2})$. Hence $\pi_0 = (n+1, 1, \ldots, 1)$. Also, $\pi_1$ is of length $\frac{p-1}{2}$ and has degree sequence $n_{\pi_1} = (1, 2, \ldots, \frac{p-1}{2})$. So $\pi_1 = (1, 1, \ldots, 1)$. From Theorem~\ref{thm:value2weight},
\begin{equation*}
\chi^\lambda(2^w) = \binom{n+p}{\frac{p-1}{2}}F_{\pi_0}F_{\pi_1}.
\end{equation*}

As $n_{\pi_0} = (1, 2, \ldots, \frac{p-1}{2}, n+\frac{p+1}{2})$,
\begin{equation*}
F_{\pi_0} = \frac{(n+\frac{p+1}{2})!}{(n+\frac{p+1}{2})!(\frac{p-1}{2})!(\frac{p-3}{2})!\ldots 1!}\Delta(n_{\pi_0}).
\end{equation*}

As $\Delta(n_{\pi_0}) = (n + \frac{p-1}{2})\ldots(n+1)\Delta(1, 2, \ldots, \frac{p-1}{2})$, $p$ does not divide $F_{\pi_0}$.

Furthermore, $\pi_1$ is a single column, so $F_{\pi_1} = F_{\pi_1'} = 1$. Then $p$ does not divide $F_{\pi_1}$ either.

Finally,
\begin{equation*}
\binom{n+p}{\frac{p-1}{2}} = \frac{(n+p)(n+p-1)\ldots(n+\frac{p+1}{2}+1)}{(\frac{p-1}{2})!}.
\end{equation*}

Putting this together gives $\nu_p(\chi^\lambda(2^w)) = \nu_p(n+p) = 1$, where the last equality follows from the fact that $p^2 \mid n$.

Let $k := w - p = n$. We now compute $\nu_p(\chi^\lambda(2^k1^{2n+2p-2k}))$. From the Murnaghan-Nakayama formula, we must consider all recursive ways of removing $k$ $2$-hooks from $\lambda$. Because $k = n$, we will never be in position to remove a $2$-hook touching the first $2$ cells of the first row. So in any such sequence there will be $k - i$ horizontal $2$-hooks removed from the first row, and $i$ $2$-hooks removed from the $(p-1) \times 2$ block of the Ferrers diagram of $\lambda$. Furthermore, $i \leq p-1$, as at most $p-1$ $2$-hooks can be removed from the $(p-1) \times 2$ block. This analysis allows us to write the following:
\begin{align*}
\chi^\lambda(2^n1^{2p}) &= \sum_{\mu \leq_n \lambda} \delta_2(\lambda, \mu)F_{2, \lambda/\mu}F_\mu,\textrm{ by Theorem \ref{thm:valuelessthan2weight}},\\ 
&=\sum_{i=0}^{p-1} \sum_{\substack{\mu \leq_n \lambda\\ \mu_1 = 2+2i}} \delta_2(\lambda, \mu)F_{2, \lambda/\mu}F_\mu,\\
&= \sum_{i=0}^{p-1} \sum_{\substack{\mu \leq_n \lambda\\ \mu_1 = 2+2i}} \delta_2(\lambda, \mu)\binom{n}{i}F_{2, \lambda^*/\mu^*}F_\mu,
\end{align*}
where $\lambda^*$, respectively $\mu^*$, is the partition $\lambda$, respectively $\mu$, without the first part. Here we have used the identity $F_{2, \lambda/\mu} = \binom{n}{i}F_{2, \lambda^*/\mu^*}$, which holds as $\binom{n}{i}$ counts the number of ways of choosing the $n - i$ hooks to be removed from the first row, out of the sequence of $n$. Therefore
\begin{equation*}
\chi^\lambda(2^n1^{2p}) = \sum_{i=0}^{p-1} \binom{n}{i}C_{\lambda, i},
\end{equation*}
where $C_{\lambda, i}$ is some integer.

As $i < p$, $p$ will divide $\binom{n}{i}$ for $i \geq 1$. On the other hand, for $i = 0$,
\begin{equation*}
C_{\lambda, 0} = F_{2, 2, \ldots, 2} = \frac{(2p)!}{(p+1)!p!\ldots 2!} \Delta(2, 3, \ldots, p+1).
\end{equation*}
Now $p$ does not divide $\Delta(2, 3, \ldots, p+1)$, as it is a product of numbers less than $p$. Also, $\nu_p((2p)!) = 2$ and $\nu_p((p+1)!p!\ldots 2!) = 2$. So $p$ does not divide $C_{\lambda, 0}$. Then $p \centernot\mid \chi^\lambda(2^n1^{2p})$.

We obtain a lower bound for the slope of the right-most edge of the Newton polygon by looking at the coefficients of $x^0$ and $x^{2n + 2p - 2k} = x^{2p}$. These are $(-1)^w\frac{H(\lambda)}{2^ww!}\chi^\lambda(2^w)$ and $(-1)^n\frac{H(\lambda)}{2^n(2p)!n!}\chi^\lambda(2^n1^{2p})$. So the slope is at least
\begin{equation*}
\frac{1}{2p}\left(\nu_p(n!) - \nu_p(w!) + \nu_p((2p)!) + \nu_p(\chi^\lambda(2^w)) - \nu_p(\chi^\lambda(2^n1^{2p}))\right)
\end{equation*}
But $n = w - p$ is divisible by $p^2$, hence $\nu_p((w-p)!) - \nu_p(w!) = -\nu_p(w) = -1$. Also $\nu_p((2p)!) = 2$.

Furthermore, the analysis above shows that $\nu_p(\chi^\lambda(2^w)) = 1$ and $\nu_p(\chi^\lambda(2^n1^{2p})) = 0$. Therefore, the slope is at least
\begin{equation*}
\frac{-1 + 2 + 1}{2p} = \frac{1}{p}.
\end{equation*}

This means any upper bound for $m=0$ must be at least $\frac{1}{p}$, finishing the proof.
\end{proof}
\setcounter{case}{0}
\renewcommand{\thecase}{\arabic{case}}

If $\lambda$ has $2$-quotient $(\pi_0, \pi_1)$, then the condition that $p$ does not divide $\Delta(n_\lambda)$ in Theorem~\ref{thm:slope} can be relaxed to $p$ does not divide $\Delta(n_{\pi_0})\Delta(n_{\pi_1})$.

\section{Schur congruences for Wronskian Hermite polynomials}
\label{sec:schur}

In order to apply Lemma~\ref{lem:filaseta}, we also need to know how to select a prime that divides the first $n - \ell - 1$ coefficients. At first sight this is a difficult problem, because we do not have a simple form for the coefficients of $\He_\lambda(x)$. We go around this obstacle by proving a congruence theorem for Wronskian Hermite polynomials.

Schur established the congruence relation for Legendre polynomials modulo an odd prime $p$: if $n$ and $m$ are two nonnegative integers with $m < p$, then
\begin{equation*}
P_{pn+m}(x) \equiv P_n(x^p)P_m(x)\mod p,
\end{equation*}
where $P_n(x)$ is the Legendre polynomial of degree $n$.

Later, Carlitz \cite{Carlitz54} extended this to Hermite and Laguerre polynomials. In particular, he proved the following theorem (stated in \cite{Carlitz54} for the classical Hermite polynomials):
\begin{theorem}[Carlitz, \cite{Carlitz54}]
\label{thm:carlitz}
For any odd $m \geq 3$ and any $n \geq 0$, we have $\He_{n+m}(x) \equiv x^m\He_n(x)\mod m$.
\end{theorem}
Carlitz proves a more general statement for sequences constructed using recurrence relations, but the proof nicely specializes to Hermite polynomials. For completeness, we are going to include a proof of Theorem~\ref{thm:carlitz}.
\begin{proof}[Proof of Theorem~\ref{thm:carlitz}]
We will prove the statement of the theorem by induction on $n \geq 0$.

It is known that the Hermite polynomials $\He_k(x)$ have the following explicit expression:
\begin{equation*}
\He_k(x) = \sum_{j=0}^{\lfloor\frac{k}{2}\rfloor} (-1)^j \frac{k!}{2^j(k-2j)!j!}x^{k-2j}.
\end{equation*}
This can be rewritten in the following way:
\begin{equation}
\label{eq:hermite_form}
\He_k(x) = \sum_{j=0}^{\lfloor\frac{k}{2}\rfloor} (-1)^j \frac{1}{2^j} \binom{k-j}{j}k(k-1)\dots(k-j+1) x^{k-2j}.
\end{equation}
Taking $k=m$ in \eqref{eq:hermite_form}, every coefficient is of the form $2^{-j}\binom{m-j}{j} m(m-1)\dots(m-j+1)$. Because $m$ is an odd integer, $2$ has an inverse modulo $m$, and so we can compute the value of the coefficients modulo $m$ by computing $2^{-j}, \binom{m-j}{j}, m(m-1)\dots(m-j+1)$ separately mod $m$, and then multiplying. But the last product is going to
be $0$ modulo $m$, unless $j=0$. So $\He_m(x) \equiv x^m\mod m$, proving the theorem for $n=0$.

In the same way we deduce that $\He_{m+1}(x) \equiv x^{m+1}\mod m$. As $x^{m+1} = x^m \He_1(x)$, the theorem is true for $n = 1$ as well.

Now assume $n > 1$. Then by the recurrence relation,
\begin{equation*}
\He_{n+m}(x) = x\He_{n-1+m}(x) - (n-1+m)\He_{n-2+m}(x).
\end{equation*}

Taking mod $m$ and using the induction hypothesis, we get 
\begin{align*}
\He_{n+m}(x) &\equiv x^{m+1}\He_{n-1}(x) - (n-1)x^m\He_{n-2}(x)\mod m \\
&\equiv x^m(x\He_{n-1}(x) - (n-1)\He_{n-2}(x))\mod m \\
&\equiv x^m\He_n(x)\mod m, \textrm{ by the recurrence relation again}.
\end{align*}
\end{proof}

Note that the requirement that $m$ is an odd integer is unfortunately necessary. For example, $\He_4(x) \equiv x^4+1\mod 2$, and not $x^4$, as would be predicted by the theorem.

We will now generalize this result to Wronskian Hermite polynomials.
\begin{proof}[Proof of Theorem~\ref{thm:hermitemodp}]
Consider the integers $n_{\lambda, i} \mod m, 1 \leq i \leq r$.

First we show they are distinct. Suppose for a contradiction that $n_{\lambda, i} \equiv n_{\lambda, j} \mod m$ for some $i \neq j$. Then by definition, $n_{\lambda, j} - n_{\lambda, i}$ divides $\Delta(n_\lambda)$, and so $m | \Delta(n_\lambda)$, a contradiction to our assumption that $m$ and $\Delta(n_\lambda)$ are coprime.

Therefore these integers form the degree sequence of a partition $\mu$. Note that some of the parts of $\mu$ may be $0$, which is allowed by our definition.

We slightly abuse the notation to define $n_{\mu, i} := n_{\lambda, i} \mod m$ for $1 \leq i \leq r$. We will consider $n_\mu$ to be the degree sequence of $\mu$, even though these numbers are not necessarily in increasing order.

Set $m_i = n_{\lambda, i} - n_{\mu, i}$ for $1 \leq i \leq r$. For any $1 \leq i \leq r$ and any $j \geq 0$, we will show that
\begin{equation}
\label{eq:deriv}
\He^{(j)}_{n_{\lambda,i}}(x) \equiv x^{m_i}\He^{(j)}_{n_{\mu,i}}(x) \mod m.
\end{equation}

By Theorem \ref{thm:carlitz}, $\He_{n_{\lambda,i}}(x) \equiv x^{m_i}\He_{n_{\mu,i}}(x) \mod m$. This proves \eqref{eq:deriv} in the case $j=0$.

To establish the claim for $j>0$, note that
\begin{equation}
\label{eq:hermite_deriv}
\He^{(j)}_{n_{\lambda,i}}(x) = n_{\lambda,i}(n_{\lambda,i} - 1)\ldots(n_{\lambda,i} - j +1)\He_{n_{\lambda,i} -j}(x).
\end{equation}

Assume first $j \leq n_{\mu,i}$. Then we can write this identity as
\begin{align*}
\He^{(j)}_{n_{\lambda,i}}(x) &= (m_i+n_{\mu,i})(m_i+n_{\mu,i} - 1)\ldots(m_i+n_{\mu,i}-j+1)\He_{m_i+n_{\mu,i}-j}(x) \\
&\equiv n_{\mu, i}(n_{\mu,i} - 1)\ldots(n_{\mu,i}-j+1)x^{m_i}\He_{n_{\mu,i}-j}(x)\mod m, \textrm{ by Theorem \ref{thm:carlitz}}, \\
&\equiv x^{m_i}\He^{(j)}_{n_{\mu,i}}(x)\mod m.
\end{align*}
So in this case the claim is true.

Assume now $j > n_{\mu, i}$. Then $\He^{(j)}_{n_{\mu,i}} = 0$. On the other hand, the product $n_{\lambda,i}(n_{\lambda,i} - 1)\ldots(n_{\lambda,i} - j +1)$ appearing in \eqref{eq:hermite_deriv} contains the term $m_i$, which is $0$ modulo $m$. So both sides of \eqref{eq:deriv} will be $0$. This proves \eqref{eq:deriv}.

Furtermore, as $m$ is coprime with $\Delta(n_\lambda)$, we have
\begin{equation}
\label{eq:divisor}
\frac{1}{\Delta(n_\lambda)} \equiv \frac{1}{\Delta(n_\mu)} \mod m.
\end{equation}
Then
\begin{align*}
\He_\lambda(x) &= \frac{\Wr[\He_{n_{\lambda, 1}}, \He_{n_{\lambda, 2}}, \ldots, \He_{n_{\lambda, r}}]}{\Delta(n_\lambda)} \\
&\equiv \frac{(x^{m_1}x^{m_2}\ldots x^{m_r})\Wr[\He_{n_{\mu, 1}}, \He_{n_{\mu, 2}}, \ldots, \He_{n_{\mu, r}}]}{\Delta(n_\mu)} \mod m, \textrm{by \eqref{eq:deriv} and \eqref{eq:divisor}}, \\
&\equiv x^{|\lambda| - |\mu|}\He_\mu(x)\mod m,
\end{align*}
as the definition of $\He_\mu(x)$ is invariant under a permutation of the degree sequence. This proves the theorem.
\end{proof}

\section{Proof of Theorem~\ref{thm:main1}}
\label{sec:main1}

Our applications of Filaseta's lemma can be synthetised in the following way.
\begin{lemma}
\label{lem:scenario}
Let $\lambda$ be a partition with at most $2$ parts, and $p$ a prime such that $p > \max\{2, |\bar{\lambda}|\}$ and $p \centernot\mid \Delta(n_\lambda)$. Assume $\He_\lambda(x) \equiv x^tQ(x) \mod p$ with $Q(0) \neq 0$. Then $R_\lambda(x)$ does not have a factor with degree in the interval $[\deg(Q)+1, p - |\bar{\lambda}|-1]$.
\end{lemma}
\begin{proof}
Let $s := |\bar{\lambda}|$. Then $\He_\lambda(x) = x^sR_\lambda(x)$ and so $R_\lambda(x) \equiv x^{t-s}Q(x) \mod p$. Hence $p$ divides the coefficient of $x^i$ in $R_\lambda(x)$ for $0 \leq i \leq t-s-1 = \deg(R_\lambda)-\deg(Q)-1$.

Write $s = \binom{m+1}{2}$ with $m \geq 0$. Because $\lambda$ has at most $2$ parts, $m \leq 2$. Also, $2m - 1 = s$ if $m \in \{1, 2\}$. Hence from Theorem~\ref{thm:slope}, the Newton polygon of $R_\lambda(x)$ with respect to $p$ has the slope of the right-most edge strictly less than $\frac{1}{p-s-1}$. Then from Lemma~\ref{lem:filaseta} with $\ell:=\deg(Q)$ and $k:=p-s-1$, we obtain that $R_\lambda(x)$ does not have a factor with degree in the interval $[\deg(Q)+1, p-s-1]$.
\end{proof}

We begin by showing that $R_{n, n}(x)$ does not have factors of small degree. We shall need the following result, which we established using a computer program.

\begin{lemma}[\cite{GrosuGit}]
\label{lem:specialsearch2}
If $\lambda = (n, n)$ is a partition with $n \leq 1000$ or $n \in \{2401, 4375\}$, then $R_\lambda(x)$ is irreducible.
\end{lemma}

\begin{lemma}
\label{lem:nn_up_to_10}
Let $n \geq 1000$ such that $n+1$ is a square. Then $R_{n, n}(x)$ does not have a factor of degree at most $10$.
\end{lemma}
\begin{proof}
Due to Lemma~\ref{lem:specialsearch2}, we can assume that $n \notin \{2401, 4375\}$.

Let $\lambda := (n, n)$. Then $|\bar{\lambda}| = 0$ and $\Delta(n_\lambda)=1$.

By Theorem~\ref{thm:nairshorey}, the fact that $n \geq 1000$, and $n-1 \notin \{2400, 4374\}$, the product $(n-1)n$ has a prime factor $q \geq 11$.

If $q \mid n$, then by Theorem~\ref{thm:hermitemodp}, $\He_{n, n}(x) \equiv x^{2n}\He_{0, 0}(x) \equiv x^{2n} \mod q$. By Lemma~\ref{lem:scenario}, $R_{n, n}(x)$ does not have a factor of degree at most $q - 1 \geq 10$, proving the statement in this case.

So we may assume that $q \mid n-1$. Then by Theorem~\ref{thm:hermitemodp}, $\He_{n, n}(x) \equiv x^{2(n-1)}\He_{1, 1}(x) \equiv x^{2(n-1)}(x^2+1) \mod q$. By Lemma~\ref{lem:scenario}, $R_{n, n}(x)$ does not have a factor with degree in the interval $[3, 10]$.

If $n$ has a prime factor $p \geq 3$, the argument above and Lemma~\ref{lem:scenario} imply that $R_{n, n}(x)$ does not have a factor of degree at most $2$, proving the statement. So we may assume that $n = 2^r$ for some $r \geq 1$.

We are now going to use the hypothesis that $n+1$ is a square. As $n = 2^r$, we must have $n+1 = m^2$ for some $m \geq 1$. But then $2^r = m^2-1 = (m-1)(m+1)$, showing that $m-1=2^a$ and $m+1=2^b$ for some $b > a \geq 0$. The only solution is $a=1, b=2$. Then $n=8$, contradicting our assumption that $n \geq 1000$. This shows that under our hypotheses $n$ can not be a power of $2$, finishing the proof.
\end{proof}

We will now deal with medium degree factors.

\begin{lemma}
\label{lem:nn_from_11_to_q}
Let $n \geq 1000$ and suppose $q$ is the largest prime with $q \leq n$. Then $R_{n, n}(x)$ does not have a factor with degree in the interval $[11, q-1]$.
\end{lemma}
\begin{proof}
Let $11 \leq d \leq q-1$. We will show that $R_{n, n}(x)$ does not have a factor of degree $d$.

First assume $d \leq \frac{2n-1}{3}$.

Set $k := \left\lfloor \frac{d+1}{2} \right\rfloor$. Then $k \geq 6$ and $n-k+1 \geq \frac{2(n+1)}{3} > 100$. Consequently by Theorem~\ref{thm:nairshorey}, the product $(n-k+1)(n-k+2)\ldots n$ has a prime factor $p \geq \min\{4.42k, n-k+1\}$.

As $4.42k > k$ and $n-k+1>k$, the prime $p$ divides exactly one of the numbers $n-k+1,n-k+2,\ldots,n$. Suppose $p \mid n - i$ with $0 \leq i \leq k-1$. Then $n \equiv i \mod p$. Hence by Theorem~\ref{thm:hermitemodp}, $\He_{n, n}(x) \equiv x^{2(n-i)}\He_{i, i}(x) \mod p$. Furthermore, $\deg(\He_{i, i}(x)) = 2i \leq 2(k-1)$.

Then by Lemma~\ref{lem:scenario}, $R_{n,n}(x)$ does not have a factor with degree in the interval $[2(k-1)+1, p-1]$. But $4.42k-1 \geq 4.42 (\frac{d}{2}) -1 \geq d$ and similarly, $n-k \geq n - \frac{d+1}{2} \geq d$, as $d \leq \frac{2n-1}{3}$. Therefore $p-1 \geq d$. Also $2(k-1) +1 \leq d$. Hence $R_{n,n}(x)$ does not have a factor of degree $d$.

So we may assume that $d > \frac{2n-1}{3}$. We will now use the prime $q$. By Theorem~\ref{thm:hermitemodp}, we have $\He_{n, n}(x) \equiv x^{2q}\He_{n-q, n-q}(x) \mod q$. Furthermore, $\deg(\He_{n-q, n-q}(x)) = 2(n-q)$.

By Theorem~\ref{thm:nagura} and the fact that $n \geq 1000$, we see that $q \geq \frac{5n}{6}$. Then $2(n-q) + 1 \leq \frac{n}{3}+1 \leq d$. By assumption, $d \leq q-1$. Then by Lemma~\ref{lem:scenario} applied with $Q(x) := \He_{n-q, n-q}(x)$, $R_{n,n}(x)$ does not have a factor with degree $d$. This finishes the proof.
\end{proof}

Unfortunately Lemma~\ref{lem:nn_from_11_to_q} does not rule out factors of degree close to $n$. For this, a different argument will be needed.

\begin{lemma}
\label{lem:nn_last_coeff}
Let $n \geq 1$. Then
\begin{equation*}
\He_{n, n}(0) =  \left\{
\begin{array}{ll}
\frac{n!^2}{2^{n-1}(\frac{n-1}{2})!^2}, & \textrm{if $n$ is odd,}\\
(n+1)\frac{n!^2}{2^{n}(\frac{n}{2})!^2}, & \textrm{if $n$ is even.}
\end{array}
\right.
\end{equation*}
In particular, if $n$ is odd, $\He_{n, n}(0) = r^2$ for some $r \geq 1$, while if $n$ is even, $\He_{n, n}(0) = (n+1)r^2$ for some $r \geq 1$.
\end{lemma}
\begin{proof}
Let $\lambda := (n, n)$. From Theorem~\ref{thm:bonneux},
\begin{equation*}
\He_{n, n}(0) = (-1)^n\frac{H(\lambda)}{2^nn!}\chi^\lambda(2^n) \eqByRef{eq:hnfactorial} (-1)^n\frac{(n+1)!}{2^n}\chi^\lambda(2^n).
\end{equation*}

First assume $n$ is odd. Then $\lambda$ has $2$-quotient $\pi_0=(\frac{n+1}{2})$ and $\pi_1=(\frac{n-1}{2})$. By Theorem~\ref{thm:value2weight},
\begin{equation*}
\chi^\lambda(2^n) = \delta_2(\lambda)\binom{n}{\frac{n+1}{2}}.
\end{equation*}
Note that on the $2$-abacus for $\lambda$, we have two beads on positions $n$ and $n+1$. Sliding them to the top puts them in positions $1$ and $0$. So the natural numbering is inverted, hence $\delta_2(\lambda)=-1$. Then
\begin{equation*}
\He_{n, n}(0) = (-1)^{n+1}\frac{(n+1)!}{2^n}\binom{n}{\frac{n+1}{2}} = \frac{(n+1)!n!}{2^n(\frac{n+1}{2})!(\frac{n-1}{2})!} = \frac{n!^2}{2^{n-1}(\frac{n-1}{2})!^2}.
\end{equation*}
As $n$ is odd, $2^{n-1}$ is a square. Then $\He_{n, n}(0)$ is a square.

Now assume $n$ is even. Then $\lambda$ has $2$-quotient $\pi_0=(\frac{n}{2})$ and $\pi_1=(\frac{n}{2})$. By Theorem~\ref{thm:value2weight},
\begin{equation*}
\chi^\lambda(2^n) = \delta_2(\lambda)\binom{n}{\frac{n}{2}}.
\end{equation*}
Note that on the $2$-abacus for $\lambda$, we have two beads on positions $n$ and $n+1$. Sliding them to the top maintains the natural numbering, hence $\delta_2(\lambda)=1$. Then
\begin{equation*}
\He_{n, n}(0) = (-1)^{n}\frac{(n+1)!}{2^n}\binom{n}{\frac{n}{2}} = \frac{(n+1)!n!}{2^n(\frac{n}{2})!^2} = (n+1)\frac{n!^2}{2^{n}(\frac{n}{2})!^2}.
\end{equation*}
As $n$ is even, $2^n$ is a square. Then $\He_{n, n}(0)$ is $n+1$ multiple a square.
\end{proof}

We will now show that factors with high degree can not occur.

\begin{lemma}
\label{lem:high_degree}
Let $n \geq 3$ such that $n+1$ is a square, and suppose $q \leq n$ is a prime with $q > \frac{2n}{3}$ and $q \equiv 3 \mod 4$. Then $R_{n, n}(x)$ does not have a factor with degree in the interval $[q, n]$.
\end{lemma}
\begin{proof}
Assume for a contradiction that $R_{n, n}(x) = f(x)g(x)$ with $q \leq \deg(f) \leq n$.

Set $d := \deg(f)$ and let $f(x) = \sum_{i=0}^{d}a_ix^i, g(x) = \sum_{i=0}^{2n-d}b_ix^i$. Further write $R_{n,n}(x) = \sum_{i=0}^{2n}c_ix^i$.

By assumption, $q > \frac{2n}{3} \geq 2$. Then from Lemma~\ref{lem:nn_last_coeff}, $\nu_q(c_0) = 2$.

In $R_{n,n}(x) = \He_{n, n}(x)$ only the even powers of $x$ have non-zero coefficients. Then $c_q = 0$.

Furthermore, from Theorem~\ref{thm:hermitemodp}, $\He_{n, n}(x) \equiv x^{2q}\He_{n-q, n-q}(x) \mod q$. Then $c_{2q}$ is equal to the first coefficient of $\He_{n-q, n-q}(x)$ modulo $q$. But from Lemma~\ref{lem:nn_last_coeff} and the fact that $n-q+1 < q$, the first coefficient of $\He_{n-q, n-q}(x)$ is non-zero modulo $q$. Then $\nu_q(c_{2q}) = 0$.

We now study the Newton polygon of $R_{n,n}(x)$ with respect to $q$. From Theorem~\ref{thm:slope}, (i), and the fact that $n < q^2$, the slope of the right-most edge is at most $\frac{1}{q}$. On the other hand, by looking at the coefficients of $x^0$ and $x^{2q}$, it is at least $\frac{\nu_q(c_0) - \nu_q(c_{2q})}{2q} = \frac{1}{q}$. So the slope must be exactly $\frac{1}{q}$. Then the Newton polygon looks like this: the edge from $(0, 0)$ to $(2(n-q), 0)$, followed by the edge from $(2(n-q), 0)$ to $(2n, 2)$. The second edge has only one interior point with integer coordinates: $(2n-q, 1)$.

Now Lemma~\ref{lem:dumas} tells us that the Newton polygon for $R_{n, n}(x)$ is formed by translates of the edges of the Newton polygons for $f(x)$ and $g(x)$. As $\deg(g) \geq \deg(f) \geq q > 2(n-q)$, both $f(x)$ and $g(x)$ must have an edge of positive slope in their polygons. Then necessarily the Newton polygon for $f(x)$ is the edge from $(0, 0)$ to $(d-q, 0)$, followed by the edge from $(d-q, 0)$ to $(d, 1)$. Similarly, the Newton polygon for $g(x)$ is the edge from $(0, 0)$ to $(2n-d-q, 0)$, followed by the edge from $(2n-d-q, 0)$ to $(2n-d, 1)$.

This implies that $\nu_q(a_q) = \nu_q(b_q) = 0$ and $\nu_q(a_i) \geq 1, \nu_q(b_i) \geq 1, 0 \leq i < q$. Also, $\nu_q(a_0) = \nu_q(b_0) = 1$.

Note that $a_0b_0 = c_0$. Furthermore

\begin{align*}
0 = c_q = \sum_{i=0}^q a_ib_{q-i} &= a_0b_q + a_qb_0 + \sum_{i=1}^{q-1}a_ib_{q-i}\\
&= q(\frac{a_0}{q}b_q + a_q\frac{b_0}{q}) + q^2\sum_{i=1}^{q-1}\frac{a_ib_{q-i}}{q^2}.
\end{align*}

Then
\begin{equation}
\label{eq:coeff_0q_congr}
\frac{a_0}{q}b_q + a_q\frac{b_0}{q} \equiv 0 \mod q.
\end{equation}

On the other hand,
\begin{align*}
c_{2q} &= \sum_{i=0}^{2q} a_ib_{2q-i} = a_qb_q + \sum_{i=1}^q(a_{q-i}b_{q+i}+a_{q+i}b_{q-i})\\
&= a_qb_q + q\sum_{i=1}^q(\frac{a_{q-i}}{q}b_{q+i}+a_{q+i}\frac{b_{q-i}}{q}).
\end{align*}

So
\begin{equation*}
c_{2q} \equiv a_qb_q \mod q.
\end{equation*}

Now multiply \eqref{eq:coeff_0q_congr} with the integer $\frac{b_0}{q}b_q$, to obtain
\begin{equation*}
\frac{a_0}{q}\frac{b_0}{q}b_q^2 + a_qb_q\left(\frac{b_0}{q}\right)^2 \equiv 0 \mod q.
\end{equation*}

This is equivalent modulo $q$ to 
\begin{equation}
\label{eq:square_relation}
\frac{c_0}{q^2}b_q^2 + c_{2q}\left(\frac{b_0}{q}\right)^2 \equiv 0 \mod q.
\end{equation}

We now consider two cases, depending on the parity of $n$.

If $n$ is odd, then Lemma~\ref{lem:nn_last_coeff} shows that $c_0=r^2$, for some $r \geq 1$. Furthermore, $n-q$ is even, so $c_{2q} \mod q$, which is $\He_{n-q, n-q}(0) \mod q$, is equal to $(n-q+1)t^2 \mod q$, for some $t \geq 1$. Re-arranging terms in \eqref{eq:square_relation} gives
\begin{equation*}
(n-q+1)\left(\frac{t b_0}{r b_q}\right)^2 \equiv -1 \mod q.
\end{equation*}
But $n-q+1 \equiv n+1 \mod q$, which is a square by hypothesis. Then $-1$ is a square modulo $q$. This is a contradiction with our assumption that $q \equiv 3 \mod 4$.

If $n$ is even, then Lemma~\ref{lem:nn_last_coeff} shows that $c_0=(n+1)r^2$, for some $r \geq 1$, while $c_{2q} \equiv t^2 \mod q$, for some $t \geq 1$. Then from \eqref{eq:square_relation} and the fact that $q \centernot\mid n+1$ we obtain
\begin{equation*}
(n+1)\left(\frac{r b_q}{t b_0}\right)^2 \equiv -1 \mod q.
\end{equation*}
As $n+1$ is a square, this is again a contradiction with our assumption that $q \equiv 3 \mod 4$.
\end{proof}

We are finally ready to prove Theorem~\ref{thm:main1}.

\begin{proof}[Proof of Theorem~\ref{thm:main1}]
From Lemma~\ref{lem:specialsearch2}, we may assume that $n \geq 1000$.

We use Theorem~\ref{thm:cullinan} with $x := \frac{n}{1.048}$ to obtain a prime $q \in (x, n]$ with $q \equiv 3 \mod 4$. Lemmas~\ref{lem:nn_up_to_10} and~\ref{lem:nn_from_11_to_q} imply that $R_{n, n}(x)$ does not have a factor with degree at most $q-1$. Then Lemma~\ref{lem:high_degree} shows that a factor with degree in the interval $[q, n]$ can not exist either.

This proves the theorem.
\end{proof}

\section{Proof of Theorem~\ref{thm:main2}}
\label{sec:main2}

We start with a general result that rules out factors of high degree for partitions $\lambda = (\lambda_1, \lambda_2)$ with $\lambda_2$ small. We shall need the following lemma, which we established using a computer program.

\begin{lemma}[\cite{GrosuGit}]
\label{lem:specialsearch}
If $\lambda = (n, m)$ is a partition with $m \leq 2$, and $n \leq 1000$ or $n \in \{2400, 2402, 4095, 4374, 4376,\allowbreak 7202, 13124\}$, then $R_\lambda(x)$ is irreducible.
\end{lemma}

\begin{lemma}
\label{lem:high_degree_for_nm}
Let $m \geq 1$ and define $d(m) := \lceil \frac{8.84m+4}{3.42}\rceil$. If $n \geq \max\{m, 1000\}$ and $\lambda = (n, m)$ then $R_\lambda(x)$ does not have a factor with degree in the interval $[d(m), \frac{n+m-|\bar{\lambda}|}{2}]$.
\end{lemma}
\begin{proof}
Let $s := |\bar{\lambda}|$. Note that we have $\Delta(n_\lambda) = n+1-m$.

Let $d \in [d(m), \frac{n+m-s}{2}]$ arbitrary. We will show that $R_\lambda(x)$ does not have a factor of degree $d$.

Set $k := d-2m$. It is easily checked that $d(m) \geq 2(m+1)$, hence $k \geq 2$. Furthermore, $k=2$ only if $m=1$ and $d=d(1)=4$. Also,
\begin{equation*}
n-m-(k-1) = n+m+1-d \geq \frac{n+m+s}{2} + 1 > 500.
\end{equation*}
It follows that we can apply Theorem~\ref{thm:nairshorey} to find a prime $p$ dividing the product $(n-m-(k-1))(n-m-(k-2))\ldots(n-m)$, with $p \geq \min\{4.42k, n-m-(k-1)\}$. Indeed, as $n-m-(k-1) > 500$, the only special cases in Theorem~\ref{thm:nairshorey} that we need to consider are for $k=2$, but then $m=1$ and $n \in \{2402, 4376\}$. Lemma~\ref{lem:specialsearch} implies that in this case $R_\lambda(x)$ is irreducible, so there is nothing to prove.

As $d \geq d(m) > \frac{8.84 m}{3.42}$, we have $4.42k > m + k$. Also, $d < \frac{n+1}{2} + m$, hence $n-m-(k-1) > m+k$. Therefore $p > m+k$. As $p$ divides a number between $n-m-(k-1)$ and $n-m$, the next number it will divide will be at least $n-m-(k-1) + p > n+1$. Then $p$ does not divide $\Delta(n_\lambda) = n-m+1$.

Suppose $n \equiv t \mod p$. Then $t \leq m+k-1$. Also $n+1 \equiv t+1 \mod p$, as $p$ does not divide $n+1$.

The degree sequence of $\lambda$ is $(m, n+1)$. From Theorem~\ref{thm:hermitemodp}, $\He_{n, m}(x) \equiv x^{n-t}\He_\mu(x) \mod p$, where $\mu$ has the degree sequence given by $(m, t+1)$, arranged in increasing order. Furthermore, $\deg(\He_\mu) = m+t \leq 2m+k-1 = d-1$.

From Lemma~\ref{lem:scenario} we obtain that $R_\lambda(x)$ does not have a factor of degree in the interval $[\deg(\He_\mu)+1, p-s-1]$. As $\deg(\He_\mu)+1 \leq d$, to finish the proof it is enough to show that $p-s-1 \geq d$.

If $p \geq 4.42k$, then
\begin{equation*}
p-s-1 \geq 4.42(d-2m)-4 = 4.42d - 8.84m - 4 \geq d,
\end{equation*}
since $s \leq 3$ and $d \geq d(m) = \lceil \frac{8.84m+4}{3.42}\rceil$.

If $p \geq n-m-k+1$, then
\begin{equation*}
p-s-1 \geq n+m+1-d-s-1 \geq d,
\end{equation*} 
since $d \leq \frac{n+m-s}{2}$.

Thus $p-s-1 \geq d$. Then $R_\lambda(x)$ does not have a factor of degree $d$.
\end{proof}

In comparison to Lemma~\ref{lem:high_degree_for_nm}, ruling out factors of small degree is much more difficult. We do this next for $\lambda = (n, 1)$.
\begin{lemma}
\label{lem:factor7forn1}
For any $n \geq 1000$, the polynomial $R_{n, 1}(x)$ does not have a factor of degree at most $7$.
\end{lemma}
\begin{proof}
Let $\lambda := (n, 1)$. Then $\Delta(n_\lambda) = n$ and $|\bar{\lambda}| \leq 3$. We first make $2$ observations.

\begin{claim}
\label{claim:pdividesn+1}
Let $p > \max\{2, |\bar{\lambda}|\}$ be a prime factor of $n+1$. Then $R_{n, 1}(x)$ does not have a factor of degree at most $p - |\bar{\lambda}| - 1$.
\end{claim}
\begin{proof}
From Theorem~\ref{thm:hermitemodp}, $\He_{n, 1}(x) \equiv x^{n+1}\He_{0, 0}(x) \equiv x^{n+1} \mod p$. Then we can apply Lemma~\ref{lem:scenario} with $Q(x):=1$ to conclude that $R_{n, 1}(x)$ does not have a factor of degree at most $p - |\bar{\lambda}| - 1$.
\end{proof}

\begin{claim}
\label{claim:pdividesn-2}
Let $p > \max\{2, |\bar{\lambda}|\}$ be a prime factor of $n-2$. Then $R_{n, 1}(x)$ does not have a factor of degree at most $p - |\bar{\lambda}| - 1$.
\end{claim}
\begin{proof}
As $p \geq 3$, $p  \centernot\mid n = \Delta(n_\lambda)$. Using Claim~\ref{claim:pdividesn+1} we can also assume that $p$ does not divide $n+1$.

From Theorem~\ref{thm:hermitemodp}, $\He_{n, 1}(x) \equiv x^{n-2}\He_{2, 1}(x) \equiv x^{n+1} \mod p$, as $\He_{2, 1}(x) = x^3$. Then we can apply Lemma~\ref{lem:scenario} with $Q(x):=1$ to conclude that $R_{n, 1}(x)$ does not have a factor of degree at most $p - |\bar{\lambda}| - 1$.
\end{proof}

By Lemma~\ref{lem:specialsearch}, we may assume that $n \notin \{2402, 4376\}$, otherwise $R_{n, 1}(x)$ is irreducible and there is nothing to show.

By Theorem~\ref{thm:nairshorey}, the fact that $n \geq 1000$, and $n-2 \notin \{2400, 4374\}$, the product $(n-2)(n-1)$ has a prime factor $q$ with $q > 8.84$. Hence $q \geq 11$.

If $q \mid n-2$, then from Claim~\ref{claim:pdividesn-2}, $R_{n, 1}(x)$ does not have a factor of degree at most $q-|\bar{\lambda}|-1$. As $q \geq 11$ and $|\bar{\lambda}| \leq 3$, $q-|\bar{\lambda}| - 1 \geq 7$, proving the lemma.

If $q \mid n-1$, then from Theorem~\ref{thm:hermitemodp}, $\He_{n, 1}(x) \equiv x^{n-1}\He_{1, 1}(x) \mod q$. But $\He_{1, 1}(x) = x^2+1$. Then Lemma~\ref{lem:scenario} implies that $R_{n, 1}(x)$ does not have a factor of degree in the interval $[3, 7]$.

We will now show that a factor of degree at most $2$ can not exist.

If $n$ is odd, then $|\bar{\lambda}| = 0$. Furthermore, $n-2$ is also odd. Applying Claim~\ref{claim:pdividesn-2} to any prime $p \geq 3$ dividing $n-2$, we see that $R_{n, 1}(x)$ does not have a factor of degree at most $2$.

So we may assume that $n$ is even. Then $|\bar{\lambda}| = 3$. If $n-2$ has a prime factor at least $7$, Claim~\ref{claim:pdividesn-2} implies that $R_{n, 1}(x)$ does not have a factor of degree at most $3$. So we may assume that $n-2$ has at most $2, 3$ and $5$ as prime factors.

If $3 \mid n-2$, then $3 \mid n+1$. We apply Theorem~\ref{thm:nairshorey} to the product $\frac{n-2}{3}\frac{n+1}{3}$ to obtain a prime factor $p$ at least $11$ (this is again possible as $n \geq 1000$ and the exceptions $n \in \{7202, 13124\}$ are handled by Lemma~\ref{lem:specialsearch}). Then $p$ must divide $n+1$. So from Claim~\ref{claim:pdividesn+1}, $R_{n, 1}(x)$ does not have a factor of degree at most $2$.

So we may assume that $3 \centernot\mid n-2$. If $5 \mid n-2$, then $n+1$ is not divisible by $3$ and $5$. As $n+1$ is odd, it has a prime factor $p \geq 7$. Then from Claim~\ref{claim:pdividesn+1}, $R_{n, 1}(x)$ does not have a factor of degree at most $2$.

So the only possibility left is that $n-2 = 2^r$ for some $r \geq 1$. Then $n+1=2^r+3$. Claim~\ref{claim:pdividesn+1} further implies that $n+1=5^q$ for some $q \geq 1$.

We will show that $2^r+3=5^q$ only holds for $r=q=1$.

We have that $2^r \equiv 2 \mod 5$. The powers of $2$ modulo $5$ are cyclically equal to $2, 4, 3, 1$. So $r = 4a + 1$ for some $a \geq 0$.

We may assume that $a \geq 1$, as otherwise we have the solution $r=q=1$. Then $2^r+3 = 2 \cdot 16^a + 3 \equiv 3 \mod 8$. However, the powers of $5$ modulo $8$ are cyclically equal to $5$ and $1$. So modulo $8$ the two sides of the equality are different, a contradiction for $a \geq 1$.

This means $n = 4$, contradicting our assumption that $n \geq 1000$. This finishes the proof.
\end{proof}

We are now ready to prove Theorem~\ref{thm:main2}.

\begin{proof}[Proof of Theorem~\ref{thm:main2}]
From Lemma~\ref{lem:specialsearch}, we may assume that $n \geq 1000$.

Let $\lambda := (n, 1)$ and set $s:=|\bar{\lambda}|$. If $R_\lambda(x)$ admits a factorization into polynomials of smaller degree, then at least one of the polynomials has degree at most $\frac{\deg(R_\lambda(x))}{2} = \frac{n+1-s}{2}$. Therefore it is enough to show that $R_\lambda(x)$ does not have a factor of degree at most $\frac{n+1-s}{2}$.

From Lemma~\ref{lem:factor7forn1}, the polynomial $R_\lambda(x)$ does not have a factor of degree at most $7$. Furthermore, from Lemma~\ref{lem:high_degree_for_nm}, $R_\lambda(x)$ does not have a factor of degree in the interval $[d(1), \frac{n+1-s}{2}]$ (with $d(\cdot)$ defined as in the lemma). But $d(1) = 4$, hence $R_\lambda(x)$ does not have a factor of degree at most $\frac{n+1-s}{2}$, proving the theorem.
\end{proof}

\section{{P}roof of Theorem~\ref{thm:main3}}
\label{sec:main3}

We start by determining the character values appearing in $\He_{n, 2}(x)$.
\begin{lemma}
\label{lem:explicit_He_2n}
Let $\lambda = (n, 2)$ have $2$-weight $w$. Then
\begin{equation*}
\He_\lambda(x) = \sum_{k=0}^{w} (-1)^k\frac{H(\lambda)}{2^k(n+2-2k)!k!}\frac{(n-2k)^2+n-2}{2}x^{n+2-2k}.
\end{equation*}
\end{lemma}
\begin{proof}
If $n$ is even, the $2$-core is empty and the $2$-weight is $w = \frac{n+2}{2}$. If $n$ is odd, the $2$-core is $(1, 0)$ and the $2$-weight is $w = \frac{n+1}{2}$.

Let $k \leq w - 2$. Consider a sequence of $k$ $2$-hooks that can be recursively removed from $\lambda$. As $n - 2k \geq 2$, at every step the first row will have more squares than the second, so we will never be in position to delete a vertical $2$-hook. Hence no matter how we remove $k$ consecutive $2$-hooks from $\lambda$, all the hooks must be horizontal, i.e. with both squares in the same row of the Ferrers diagram of $\lambda$.

One of the $k$ hooks may remove the second row of the Ferrers diagram. Therefore by the Murnaghan-Nakayama rule,
\begin{equation*}
\chi^\lambda(2^k1^{n+2-2k}) = kF_{n-2(k-1), 0} + F_{n-2k, 2},
\end{equation*}
where the first term comes from the fact that out of $k$ consecutive $2$-hooks to be removed, there are $k$ ways to select one that removes the second row.

Now note that for a partition of the form $(m, 2)$ we have $F_{m, 2} = \frac{(m+2)(m-1)}{2}$. Using this and the fact that $F_{m, 0} = 1$, we get
\begin{equation*}
\chi^\lambda(2^k1^{n+2-2k}) = k + \frac{(n-2k+2)(n-2k-1)}{2} = \frac{(n-2k)^2+n-2}{2}.
\end{equation*}

For $n$ even and $k = w-1 = \frac{n}{2}$, we have
\begin{equation*}
\chi^\lambda(2^k1^{n+2-2k}) = kF_{n-2(k-1), 0} - F_{1, 1} = k - 1 = \frac{(n-2k)^2+n-2}{2},
\end{equation*}
as there is also the possibility to remove $k-1$ horizontal hooks from the first row, and a vertical hook afterwards.

For $n$ odd and $k = w - 1 = \frac{n-1}{2}$, we have
\begin{equation*}
\chi^\lambda(2^k1^{n+2-2k}) = kF_{n-2(k-1), 0} = k = \frac{(n-2k)^2+n-2}{2},
\end{equation*}
as one of the $2$-hooks must remove the second row, and all the hooks are horizontal.

For $n$ even and $k = w$, we have $\chi^\lambda(2^k) = (k-1)F_{0, 0} + 1 = w = \frac{(n-2k)^2+n-2}{2}$, as either all the hooks are horizontal, or the first $k-2$ are horizontal and the last $2$ hooks are vertical.

Finally, if $n$ is odd and $k = w$, we have $\chi^\lambda(2^k) = (k-1)F_{1, 0} = w-1 = \frac{(n-2k)^2+n-2}{2}$, as out of the $w$ $2$-hooks, only the last one may not remove the second row.

Replacing the character values in Theorem~\ref{thm:bonneux} gives the lemma.
\end{proof}

We use this to prove a slight sharpening of Theorem~\ref{thm:slope} for the prime $p=3$.
\begin{lemma}
\label{lem:slope_for_3}
Let $n \geq 3$ be odd. Then the slope of the right-most edge of the Newton polygon for $R_{n, 2}(x)$ with respect to $3$ is strictly less than $\frac{1}{2}$, if $n+1$ is divisible by $3$, and equal to $\frac{1}{2}$, otherwise.
\end{lemma}
\begin{proof}
Let $\lambda = (n, 2)$. Because $n$ is odd, $|\bar{\lambda}|=1$ and the $2$-weight $w$ is $\frac{n+1}{2}$. From Lemma~\ref{lem:explicit_He_2n},
\begin{equation*}
R_{n, 2}(x) = \sum_{k=0}^w (-1)^k \frac{H(\lambda)}{2^k(n+2-2k)!k!}\frac{(n-2k)^2+n-2}{2} x^{n+1-2k}.
\end{equation*}

Then the slope of the right-most edge of the Newton polygon for the prime $3$ is
\begin{equation*}
\max_{0 \leq k < w} \frac{1}{n+1-2k}\left\{\nu_3\left(\frac{(n-2w)^2+n-2}{(n+2-2w)!w!}\right) - \nu_3\left(\frac{(n-2k)^2+n-2}{(n+2-2k)!k!}\right)\right\}.
\end{equation*}

Write $k = w -j$. Then $n+2-2k = n+2 - 2w + 2j = 2j+1$ and $(n-2w)^2+n-2=2(w-1)$. So the slope is
\begin{equation}
\label{eq:slope_for_3}
\max_{1 \leq j \leq w} \frac{1}{2j}\left\{\nu_3((2j+1)!) + \nu_3((w-j)!) - \nu_3(w!) + \nu_3(w-1) - \nu_3((2j-1)^2+n-2)\right\}.
\end{equation}

If $3 \centernot\mid w-1$ then $\nu_3((w-j)!)-\nu_3(w!)+\nu_3(w-1) \leq -\nu_3(j!)$, as $\nu_3(\binom{w}{w-j}) \geq 0$.

If $j \geq 2$ and $3 \mid w-1$ then $\nu_3((w-j)!)-\nu_3(w!)+\nu_3(w-1) \leq \nu_3((w-j)!)-\nu_3((w-2)!) \leq -\nu_3((j-2)!)$.

Hence for $j \geq 2$, the quantity in \eqref{eq:slope_for_3} is at most
\begin{equation}
\label{eq:upper_bound_for_j2}
\frac{\nu_3((2j+1)!)-\nu_3((j-2)!)}{2j}.
\end{equation}

We use Lemma~\ref{lem:legendre} to upper-bound this value. We get
\begin{align*}
\frac{\nu_3((2j+1)!)-\nu_3((j-2)!)}{2j} &= \frac{1}{2j}\left(\frac{2j+1-\kappa_3(2j+1)}{2}-\frac{j-2-\kappa_3(j-2)}{2}\right)\\
&=\frac{1}{2j}\frac{j+3+\kappa_3(j-2)-\kappa_3(2j+1)}{2}\leq\frac{j+2+\kappa_3(j-2)}{4j},
\end{align*}
where the last inequality follows from the fact that $\kappa_3(2j+1) \geq 1$.

The condition $\frac{j+2+\kappa_3(j-2)}{4j} < \frac{1}{2}$ is equivalent to $\kappa_3(j-2) < j-2$. This is true for all $j > 4$.

For $j = 2$ and $j=3$, we return to \eqref{eq:upper_bound_for_j2} to see that it is $\frac{1}{4}$, respectively $\frac{1}{3}$.

For $j=4$, we see that \eqref{eq:upper_bound_for_j2} is $\frac{1}{2}$. However, if $3 \mid n+1$ then $3 \centernot\mid w-1$ and we have the stronger upper bound
\begin{equation*}
\frac{\nu_3(9!) + \nu_3((w-4)!) - \nu_3(w!) + \nu_3(w-1)}{8} \leq \frac{3}{8}.
\end{equation*}
Therefore for $j \geq 2$ the maximum in \eqref{eq:slope_for_3} is at most $\frac{1}{2}$, and if $3 \mid n+1$ the inequality is strict. This leaves only the case $j=1$ to be examined.

If $j=1$ then $(2j-1)^2+n-2 = 2(w-1)$. So the quantity in \eqref{eq:slope_for_3} is
\begin{equation*}
\frac{\nu_3(3!) + \nu_3((w-1)!)-\nu_3(w!)}{2} = \frac{1-\nu_3(w)}{2} = \frac{1-\nu_3(n+1)}{2}.
\end{equation*}
If $3 \mid n+1$ then this is at most $0$, and the slope is less than $\frac{1}{2}$. If $3 \centernot\mid n+1$ then this is $\frac{1}{2}$, and so the slope is exactly $\frac{1}{2}$.
\end{proof}

Armed with this we can rule out most small-degree factors of $R_{n, 2}(x)$.
\begin{lemma}
\label{lem:factor9forn2}
Let $n \geq 1000$. Then the polynomial $R_{n, 2}(x)$ does not have a factor of degree in the interval $[3, 9]$. Furthermore, if $n+1$ is not a power of $2$, then it does not have a factor of degree at most $2$ either.
\end{lemma}
\begin{proof}
Let $\lambda := (n, 2)$. Then $\Delta(n_\lambda) = n-1$ and $|\bar{\lambda}| \leq 1$. We start with an observation.

\begin{claim}
\label{claim:pdividesn+1forn2}
Let $p \geq 3$ be a prime factor of $n+1$. Then $R_{n, 2}(x)$ does not have a factor of degree at most $p - |\bar{\lambda}| - 1$.
\end{claim}
\begin{proof}
As $p \geq 3$, $p$ does not divide $\Delta(n_\lambda)$. Hence by Theorem~\ref{thm:hermitemodp}, $\He_{n, 2}(x) \equiv x^{n+1}\He_{1, 0}(x) \equiv x^{n+2} \mod p$. Using Lemma~\ref{lem:scenario}, we conclude that $R_{n, 2}(x)$ does not have a factor of degree at most $p - |\bar{\lambda}| - 1$.
\end{proof}

By Lemma~\ref{lem:specialsearch}, we may assume that $n \notin \{2400, 4374\}$, otherwise $R_{n, 2}(x)$ is irreducible and there is nothing to prove. From this and Theorem~\ref{thm:nairshorey}, the product $n(n+1)$ has a prime factor $q$ with $q \geq 11$.

If $q \mid n+1$, then Claim~\ref{claim:pdividesn+1forn2} and the fact that $q - |\bar{\lambda}| - 1 \geq 9$ show that $R_{n, 2}(x)$ has no factor of degree at most $9$.

So we may assume that $q \mid n$. Then by Theorem~\ref{thm:hermitemodp}, $\He_{n, 2}(x) \equiv x^n\He_{1, 1}(x) \equiv x^n(x^2+1) \mod q$. Applying Lemma~\ref{lem:scenario} with $Q(x):=x^2+1$, we obtain that $R_{n, 2}(x)$ does not have a factor of degree in the interval $[3, 9]$.

Now assume $n+1$ is not a power of $2$. We will show that a factor of degree at most $2$ does not exist.

By hypothesis, $n+1$ has a prime factor $p \geq 3$.

If $n$ is even then $|\bar{\lambda}| = 0$. Applying Claim~\ref{claim:pdividesn+1forn2} to $p$ shows that $R_{n, 2}(x)$ does not have a factor of degree at most $p-|\bar{\lambda}|-1 \geq 2$.

So we may assume that $n$ is odd. Then $|\bar{\lambda}| = 1$. If $p \geq 5$, then Claim~\ref{claim:pdividesn+1forn2} shows that no factor of degree at most $2$ exists. If $p=3$, then Lemma~\ref{lem:slope_for_3} shows that the slope of the right-most edge of the Newton polygon for $p$ is strictly less than $\frac{1}{2}$. Because $3 \mid n+1$, $3$ does not divide $\Delta(n_\lambda) = n-1$. Hence by Theorem~\ref{thm:hermitemodp}, $\He_{n, 2}(x) \equiv x^{n+2} \mod 3$. Then Lemma~\ref{lem:filaseta} shows that no factor of degree at most $2$ exists.
\end{proof}

To move forward we are going to need a bound on the modulus of the roots.
\begin{lemma}
\label{lem:modulus_bound}
Let $z$ be a real or purely imaginary root of $R_{n, 2}(x)$. Then $|z|^2 < 4n - 1$.
\end{lemma}
\begin{proof}
From Lemma~\ref{lem:bound_real_roots}, $|z| \leq x_{n+2}$ where $x_{n+2}$ is the largest root of $\He_{n+2}(x)$. From Corollary~\ref{cor:szego}, $x_{n+2} \leq \frac{2(n+1)}{\sqrt{n+4}}$. Then
\begin{equation*}
|z|^2 \leq \frac{4(n+1)^2}{n+4} < 4n - 1,
\end{equation*}
where the last inequality holds as $n \geq 2$.
\end{proof}

We will now rule out factors of degree at most $2$ in the remaining case when $n+1$ is a power of $2$. 

\begin{lemma}
\label{lem:case_n_2l-1}
Suppose $n \geq 1000$ and $n=2^\ell - 1$. Then $R_{n, 2}(x)$ does not have a factor of degree at most $2$.
\end{lemma}
\begin{proof}
Set $\lambda := (n, 2)$. Note that $n$ is odd and $|\bar{\lambda}|=1$. Assume for a contradiction that $R_{n, 2}(x)$ has a factor of degree at most $2$.

First we show that we can find a factor of degree exactly $2$. If $R_{n, 2}(x)$ has a factor of degree $1$ then it has an integer root $k \in \mathbb{Z}^*$. But from Observation~\ref{obs:symmetric_roots}, $-k$ is also a root of $R_{n, 2}(x)$. Then $(x-k)(x+k) = x^2-k^2$ divides $R_{n, 2}(x)$, and it is a factor of degree $2$.

Let us write $R_{n, 2}(x) = f(x)g(x)$, with $f, g \in \mathbb{Z}[x], \deg(f) = 2$ and $\deg(g)=n-1$. Then $f(x)$ and $g(x)$ are monic.

We now show the following.
\begin{claim}
\label{clm:opposite_roots}
The polynomial $f(x)$ is either $x^2-z^2$, or $x^2 + z^2$, for some $z \in \mathbb{R}^*$.
\end{claim}
\begin{proof}
Suppose for a contradiction that $f(x)$ has two roots $z_1$ and $z_2$ such that $z_1 \neq -z_2$. Define $\bar{f}(x) = (x+z_1)(x+z_2)$. Then $\bar{f}(x) = f(-x)$, hence $\bar{f}(x) \in \mathbb{Z}[x]$. As $-z_1$ and $-z_2$ are also roots of $R_{n, 2}(x)$ by Observation~\ref{obs:symmetric_roots}, distinct from $z_1$ and $z_2$, it follows that $f(x)\bar{f}(x)$ divides $R_{n, 2}(x)$. But from Lemma~\ref{lem:factor9forn2}, $R_{n, 2}(x)$ does not have a factor of degree $4$, a contradiction.

Hence there exists $z \in \mathbb{C}^*$ such that $f(x) = x^2-z^2$. If $z \in \mathbb{R}$ then we are done. If $z \notin \mathbb{R}$ then necessarily $-z = \bar{z}$, hence $z = iz'$ for some $z' \in \mathbb{R}$. Therefore $f(x) = x^2 + z'^2$.
\end{proof}

Let $p \geq 5$ be any prime factor of $n$. Then $p$ does not divide $\Delta(n_\lambda)=n-1$. We claim that $p$ does not divide $f(0)$. Indeed, otherwise from Lemma~\ref{lem:dumas} the Newton polygon of $R_{n, 2}(x)$ with respect to $p$ would have an edge of positive slope at least $\frac{1}{2}$. This would contradict the conclusion of Theorem~\ref{thm:slope} that the maximum slope is less than $\frac{1}{p-2} \leq \frac{1}{3}$.

Similarly, if $p$ is a prime factor of $n-2$, then $p>3$ (as $3$ does not divide $n-2+3=2^\ell$) and $p$ does not divide $f(0)$.

Let $m$ be the largest divisor of $n$ that is coprime with $3$. The equation $2^x-3^y = 1$ has the only solution $x=2,y=1$ (this result goes back to Levi ben Gerson in the 14\textsuperscript{th} Century). As $n \geq 1000$, $n$ is not a power of $3$ and so $m > 1$ and $m$ is odd. From the above, $m$ is coprime to $f(0)$.

Because $\Delta(n_\lambda) = n-1$, we can apply Theorem~\ref{thm:hermitemodp} to $m$. Then
\begin{equation}
\label{eq:prod_factor}
xf(x)g(x) = \He_{n, 2}(x) \equiv x^n\He_{1, 1}(x) \equiv x^n(x^2+1) \mod m.
\end{equation}

Let $f_m(x)$ and $g_m(x)$ be the images of $f(x)$ and $g(x)$ modulo $m$. From \eqref{eq:prod_factor}, $f_m(x)g_m(x) \equiv x^{n-1}(x^2+1)$. Write $g_m(x) = x^\ell h_m(x)$ with $h_m(0) \neq 0$ modulo $m$. If $\ell < n-1$ then $f(0)h_m(0) \equiv 0 \mod m$. As $(m, f(0)) = 1$ we must have $m \mid h_m(0)$, contradicting the assumption that $h_m(0) \neq 0$. So $\ell = n-1$. As $f(x)$ and $g(x)$ are monic, we have $h_m(x) = 1$. Hence $f_m(x) = x^2+1$.

Using this and Claim~\ref{clm:opposite_roots}, we can write
\begin{align*}
f(x) &= x^2 + (1+cm),\\
g(x) &= x^{n-1} + ax^{n-2} + dx^{n-3} + h(x),
\end{align*}
with $a, c, d \in \mathbb{Z}$ and $\deg(h) \leq n-4$. Multiplying them gives
\begin{equation*}
f(x)g(x) = x^{n+1} + ax^n+(1+cm+d)x^{n-1}+h'(x),
\end{equation*}
where $h'(x)$ is a polynomial of degree at most $n-2$.

The coefficient of $x^n$ in $R_{n, 2}(x)$ is $0$, so $a = 0$. From Lemma~\ref{lem:explicit_He_2n}, the coefficient of $x^{n-1}$ in $R_{n, 2}(x)$ is
\begin{equation*}
-\frac{1}{2}\frac{H(\lambda)}{n!} \frac{(n-2)(n-1)}{2} = -\frac{n+1}{n-1} \frac{(n-2)(n-1)}{2} = -\frac{(n+1)(n-2)}{2}.
\end{equation*}
Hence
\begin{equation}
\label{eq:coeff_xn-1}
1+cm+d = -\frac{(n+1)(n-2)}{2}.
\end{equation}

We will now apply Theorem~\ref{thm:hermitemodp} to $n-2$. This is possible, as $n-2$ is odd and $(n-2, \Delta(n_\lambda))=1$. We obtain
\begin{equation*}
xf(x)g(x) = \He_{n, 2}(x) \equiv x^{n-2}\He_{2, 2}(x) \equiv x^{n-2}(x^4+3) \mod n-2.
\end{equation*}
Recall that $g(x) = x^{n-1}+dx^{n-3}+h(x)$. If $h(x) \neq 0 \mod n-2$ then it has a lowest non-zero term modulo $n-2$ of the form $ex^r$. But $r \leq n-4$, and $f(x)g(x)$ has no term of degree less than $n-3$ modulo $n-2$. So $f(0)e \equiv 0 \mod n-2$. As $(n-2, f(0))=1$, it follows that $e \equiv 0 \mod n-2$, a contradiction. So $h(x) \equiv 0 \mod n-2$.

Hence $g(x) \equiv x^{n-1}+dx^{n-3} \mod n-2$. So $(1+cm)d \equiv 3 \mod n-2$. From \eqref{eq:coeff_xn-1}, $1+cm+d \equiv 0 \mod n-2$. Then
\begin{equation}
\label{eq:prod_of_zero_terms}
-(1+cm)^2 \equiv 3 \mod n-2.
\end{equation}
Now write $n = 3^ym, y \geq 0$. Multiply \eqref{eq:prod_of_zero_terms} with $3^{2y}$ to get
\begin{equation*}
-(3^y+c3^ym)^2 = -(3^y+cn)^2 \equiv -(3^y+2c)^2 \equiv 3^{2y+1} \mod n-2.
\end{equation*}
We re-arrange this into
\begin{equation}
\label{eq:c_identity}
3^{2y+1}+(3^y+2c)^2 \equiv 0 \mod n-2.
\end{equation}

From Lemma~\ref{lem:modulus_bound}, Claim~\ref{clm:opposite_roots}, and the fact that $f(0) = z^2$ for a real or purely imaginary root $z$ of $R_{n, 2}(x)$, we obtain $|f(0)| < 4n - 1$. But $|f(0)| \geq |c|m - 1$, so $|c| < \frac{4n}{m} = 4 \cdot 3^y$. Hence
\begin{equation*}
|3^{2y+1}+(3^y+2c)^2| \leq 3^{2y+1}+(3^y+2|c|)^2 < 3^{2y+1}+81\cdot 3^{2y} = 84 \cdot 3^{2y}.
\end{equation*}

The left-hand side of \eqref{eq:c_identity} is a positive integer. So in order for \eqref{eq:c_identity} to hold we must have
\begin{equation}
\label{eq:3y_inequality}
84\cdot 3^{2y} \geq n-2 = 2^\ell - 3.
\end{equation}

From Lemma~\ref{lem:nu_inequality} we get $\nu_3(n) = \nu_3(2^\ell - 1) \leq 1 + \nu_3(\ell)$. But $\nu_3(n) = y$, so $3^y \leq 3\ell$. Using this in \eqref{eq:3y_inequality} we get
\begin{equation}
\label{eq:ell_inequality}
756\ell^2 \geq 2^\ell - 3.
\end{equation}

Inequality \eqref{eq:ell_inequality} is false for $\ell \geq 18$. As $n \geq 1000$, we also have $\ell \geq 10$. So the only possible cases are $10 \leq \ell \leq 17$. Replacing these values one by one in \eqref{eq:3y_inequality} (with $y := \nu_3(2^\ell-1)$), we see that the inequality is verified only for $\ell=12$. But then $n = 4095$ and in this case $R_{n, 2}(x)$ is irreducible by Lemma~\ref{lem:specialsearch}.
\end{proof}

\begin{proof}[Proof of Theorem~\ref{thm:main3}]
From Lemma~\ref{lem:specialsearch}, we may assume that $n \geq 1000$.

Let $\lambda := (n, 2)$ and set $s:=|\bar{\lambda}|$. If $R_\lambda(x)$ admits a factorization into polynomials of smaller degree, then at least one of the polynomials has degree at most $\frac{\deg(R_\lambda(x))}{2} = \frac{n+2-s}{2}$. Therefore it is enough to show that $R_\lambda(x)$ does not have a factor of degree at most $\frac{n+2-s}{2}$.

From Lemmas~\ref{lem:factor9forn2} and \ref{lem:case_n_2l-1}, the polynomial $R_\lambda(x)$ does not have a factor of degree at most $9$. Furthermore, from Lemma~\ref{lem:high_degree_for_nm}, $R_\lambda(x)$ does not have a factor of degree in the interval $[d(2), \frac{n+2-s}{2}]$ (with $d(\cdot)$ defined as in the lemma). But $d(2) = 7$, hence $R_\lambda(x)$ does not have a factor of degree at most $\frac{n+2-s}{2}$, proving the theorem.
\end{proof}

\section{Some applications}
\label{sec:applications}

In this section we are going to apply the previous results to obtain more information about the roots of Wronskian Hermite polynomials. To put this into context, we give a short overview of what is already known about the roots of $\He_\lambda(x)$. Some of the stated results were originally formulated in terms of the classical Hermite polynomials. However, as $\He_n(x)$ is a rescaling of $H_n(x)$ (see also Lemma~\ref{lem:wronskian_vs_He} below), information about the roots translates to $\He_\lambda(x)$.

Results of Adler (\cite{Adler94}) and Krein (\cite{Krein57}) characterize the partitions $\lambda$ for which $\He_\lambda(x)$ has no real zeros (see \cite{FerreroUllate15}, Theorem~1.1).

The case of rectangular partitions $\lambda = (n, n, \ldots, n)$ is of great interest because it yields families of rational solutions to the fourth Painlev{\'e} equation. Karlin and Szeg{\H o} (\cite{KarlinSzego}) determined the number of real zeros of the Wronskian in this case.

\begin{theorem}[\cite{KarlinSzego}]
Let $\lambda = (n, n, \ldots, n)$ of even length. Then $\He_\lambda(x)$ has only complex non-real roots.
\end{theorem}

\begin{theorem}[\cite{KarlinSzego}]
Let $\lambda = (n, n, \ldots, n)$ and $\mu=(n+1, n+1, \ldots, n+1)$ have the same odd length. Then $\He_\lambda(x)$ has $n$ simple real roots and $\He_\lambda(x)$, $\He_\mu(x)$ strictly interlace.
\end{theorem}

If $f:(a,b)\to \mathbb{R}$ is a derivable function, an element $x\in (a,b)$ is a \textit{nodal zero} of $f$ if $f(x)=0$ and $f$ changes sign when going through $x$.

\begin{theorem}[\cite{KarlinSzego}]
Let $\lambda = (n, \ldots, n, k)$ of even length. Then $\He_\lambda(x)$ has $n-k-1$ nodal zeros.
\end{theorem}

Roberts (\cite{Roberts}) computed the discriminant of $\Wr[H_{n}, H_{n+1}, \ldots, H_{n+\ell-1}]$. In particular, as the discriminant is non-zero, the roots of $\He_\lambda(x)$ are simple, when $\lambda = (n, n, \ldots, n)$. Recently, Dur\'an (\cite{Duran2020}) provided a new proof for the simplicity of the roots in this case. The asymptotic behaviour of the roots for rectangular partitions was studied (see \cite{Buckingham}, \cite{Masoero}): it turns out that for $n$ and $\ell$ large, the roots densely fill a quadrilateral region.

In \cite{FerreroUllate15}, using the irreducibility of Hermite polynomials, the authors show that Conjecture~\ref{conj:veselov} holds for $\Wr[H_n, H_m]$. For $k$ fixed and $n$ large, the asymptotic distribution of the zeros of $\Wr[H_n, H_{n+k}]$ was studied in \cite{FelderHemeryVeselov12}.

Let $\varphi_n(x) := e^{-\frac{x^2}{2}}H_n(x)$ be the Hermite functions. First, we derive a relation between the Wronskian of Hermite functions and $\He_\lambda(x)$.

\begin{lemma}
\label{lem:wronskian_vs_He}
Let $\lambda$ be a partition with degree sequence $(n_1, n_2, \ldots, n_r)$. Then
\begin{equation*}
\Wr[\varphi_{n_1}(x), \varphi_{n_2}(x), \ldots, \varphi_{n_r}(x)] = e^{-\frac{rx^2}{2}}2^{\frac{|\lambda|+r(r-1)}{2}}\Delta(n_\lambda)\He_\lambda(\sqrt{2}x).
\end{equation*}
\end{lemma}
\begin{proof}
From Proposition $3.1$, \cite{FerreroUllate15}, it follows that
\begin{equation*}
\Wr[\varphi_{n_1}(x), \varphi_{n_2}(x), \ldots, \varphi_{n_r}(x)] = e^{-\frac{rx^2}{2}}\Wr[H_{n_1}(x), H_{n_2}(x), \ldots, H_{n_r}(x)].
\end{equation*}
From $(4.3)$ and $(4.4)$ in \cite{BonneuxStevens}, we can further write
\begin{equation*}
\Wr[H_{n_1}(x), H_{n_2}(x), \ldots, H_{n_r}(x)] = 2^{\frac{|\lambda|+r(r-1)}{2}}\Delta(n_\lambda)\He_\lambda(\sqrt{2}x).
\end{equation*}
Putting these two identities together proves the lemma.
\end{proof}
In particular, there is a correspondence between the roots of the Wronskian and the roots of $\He_\lambda$ given by the transformation $x \mapsto \sqrt{2}x$.

\begin{lemma}
\label{lem:shift_div}
Let $\lambda = (\lambda_1, \lambda_2, \ldots, \lambda_r)$ be a partition with degree sequence $(n_1, n_2, \ldots, n_r)$. Let $\mu_1$ be the partition with degree sequence $(n_1, n_2, \ldots, n_{r-1})$, and $\mu_2$ be the partition with degree sequence $(n_1, \ldots, n_{r-2}, n_r)$. If $R_{\mu_1}(x)$ is irreducible and $R_{\mu_1}(x) \mid R_\lambda(x)$ then $R_{\mu_1}(x) \mid R_{\mu_2}(x)$.
\end{lemma}
\begin{proof}
From Jacobi's identity for Wronskians we obtain that
\begin{equation*}
\Wr[\He_{n_1}, \ldots, \He_{n_r}]\Wr[\He_{n_1}, \ldots, \He_{n_{r-2}}] = \Wr[\Wr[\He_{n_1}, \ldots, \He_{n_{r-1}}], \Wr[\He_{n_1}, \ldots, \He_{n_{r-2}}, \He_{n_r}]].
\end{equation*}

Let $\rho$ be the partition with degree sequence $(n_1, \ldots, n_{r-2})$. Using Definition~\ref{def:wronskian} and the fact that $\He_\lambda(x) = x^{|\bar{\lambda}|}R_\lambda(x)$ we obtain
\begin{equation*}
\Delta(n_\lambda)x^{|\bar{\lambda}|}R_\lambda(x)\Delta(n_\rho)x^{|\bar{\rho}|}R_\rho(x) = \Delta(n_{\mu_1})\Delta(n_{\mu_2})\Wr[\He_{\mu_1}(x), \He_{\mu_2}(x)].
\end{equation*}

This can be rewritten as
\begin{align*}
\frac{\Delta(n_\lambda)\Delta(n_\rho)}{\Delta(n_{\mu_1})\Delta(n_{\mu_2})}x^{|\bar{\lambda}| + |\bar{\rho}|}R_\lambda(x)R_\rho(x) &= \He_{\mu_1}(x)\He_{\mu_2}'(x) - \He_{\mu_1}'(x)\He_{\mu_2}(x)\\
&= x^{|\bar{\mu}_1|}R_{\mu_1}(x)\He_{\mu_2}'(x) - \He_{\mu_1}'(x)x^{|\bar{\mu}_2|}R_{\mu_2}(x).
\end{align*}
If $R_{\mu_1}(x)$ divides $R_\lambda(x)$, then it must divide $\He_{\mu_1}'(x)x^{|\bar{\mu}_2|}R_{\mu_2}(x)$. Now
\begin{equation*}
\He_{\mu_1}'(x)x^{|\bar{\mu}_2|}R_{\mu_2}(x) = |\bar{\mu}_1|x^{|\bar{\mu}_1|-1+|\bar{\mu}_2|}R_{\mu_1}(x)R_{\mu_2}(x) + x^{|\bar{\mu}_1| + |\bar{\mu}_2|}R_{\mu_1}'(x)R_{\mu_2}(x).
\end{equation*}
As $R_{\mu_1}(x)$ is irreducible, it is coprime with $x^{|\bar{\mu}_1| + |\bar{\mu}_2|}R_{\mu_1}'(x)$. Therefore, it must divide $R_{\mu_2}(x)$.
\end{proof}

\begin{lemma}
\label{lem:equal_remainder_poly}
Let $\mu_1, \mu_2$ be partitions of length $r$ which have the same degree sequence, except possibly for the last element. If $R_{\mu_1}(x) = R_{\mu_2}(x)$ then either $\mu_1 = \mu_2$, or one is a partition of $0$, while the other is a partition of $1$.
\end{lemma}
\begin{proof}
Because the first $r-1$ elements of the degree sequences coincide, the parts of the partitions, except possibly for the first one, are equal as well. Therefore we may write $\mu_1 = (\alpha, a_1, \ldots, a_{r-1})$ and $\mu_2 = (\beta, a_1, \ldots, a_{r-1})$. Moreover, we may assume without lack of generality that $\beta \geq \alpha$.

As $R_{\mu_1}(x) = R_{\mu_2}(x)$, the subleading coefficients must be equal. From Theorem~\ref{thm:subleading_coeff} we obtain
\begin{equation*}
-\frac{1}{2}\left(\alpha(\alpha-1) + \sum_{i=1}^{r-1}a_i(a_i-2i-1)\right) = -\frac{1}{2}\left(\beta(\beta-1) + \sum_{i=1}^{r-1}a_i(a_i-2i-1)\right).
\end{equation*}
Therefore $\alpha(\alpha-1) = \beta(\beta-1)$ which has the solutions $\alpha = \beta$ and $\alpha = 1 - \beta$.

If $\alpha = \beta$ then $\mu_1 = \mu_2$, proving the lemma.

If $\alpha = 1 - \beta$, then as $\beta \geq \alpha \geq 0$, we have $\alpha = 0$ and $\beta = 1$. Hence $\mu_1 \vdash 0$ and $\mu_2 \vdash 1$, and it is clear that $R_0(x) = R_1(x) = 1$. This proves the lemma.
\end{proof}

\begin{proof}[Proof of Corollary~\ref{cor:num_real_roots}]
Let $n_\lambda = (\lambda_3, \lambda_2 + 1, \lambda_1+2)$ be the degree sequence of $\lambda$. We claim this sequence is non-degenerate. Let $1 \leq i < j \leq 3$. We check the semi-degeneracy conditions (i) and (ii) for $i$.

First assume $i = 1$.

We show (ii) first. Note that $\Wr[\varphi_{n_k}(x)] = \varphi_{n_k}(x) = e^{-\frac{x^2}{2}}H_{n_k}(x)$, where $H_{n_k}(x)$ is a power of $x$ times an irreducible polynomial. Therefore $\varphi_{n_i}(x)$ and $\varphi_{n_j}(x)$ have at most the root $x=0$ in common, verifying (ii).

Now we show (i) holds. Note that $\Wr[\varphi_{n_i}, \varphi_{n_j}] = \varphi_{n_i}\varphi_{n_j}' - \varphi_{n_i}'\varphi_{n_j}$. But $n_i = \lambda_3 \in \{1, 2\}$, and $H_1(x) = 2x, H_2(x) = 4x^2 - 2$, which are both irreducible. Then if $H_{n_i}(x)$ shares a root with $\Wr[\varphi_{n_i}, \varphi_{n_j}]$, it must divide it. Therefore $H_{n_i}(x)$ divides $\varphi_{n_j}(x)$, which is only possible if $n_i = 1$. Then the only possible root in common is $x=0$. This proves (i) for $i=1$.

Now assume $i = 2$. Then $j = 3$.

First we show (ii) holds. Suppose for a contradiction that $\Wr[\varphi_{n_1}, \varphi_{n_2}]$ and $\Wr[\varphi_{n_1}, \varphi_{n_3}]$ have a non-zero root in common. Let $\mu_1$ be the partition with degree sequence $(n_1, n_2) = (\lambda_3, \lambda_2+1)$. Similarly, let $\mu_2$ be the partition with degree sequence $(n_1, n_3) = (\lambda_3, \lambda_1+2)$. Then by Lemma~\ref{lem:wronskian_vs_He}, $\He_{\mu_1}(x)$ and $\He_{\mu_2}(x)$ have a non-zero root in common. By Theorems~\ref{thm:main2} and~\ref{thm:main3}, $R_{\mu_1}(x)$ and $R_{\mu_2}(x)$ are irreducible. As they share a root, they must be equal. Then from Lemma~\ref{lem:equal_remainder_poly}, $\mu_1 = \mu_2$, a contradiction with $\lambda_1 + 1 > \lambda_2$.
This shows (ii).

Now we show (i) holds. Suppose for a contradiction that $\Wr[\varphi_{n_1}, \varphi_{n_2}]$ and $\Wr[\varphi_{n_1}, \varphi_{n_2}, \varphi_{n_3}]$ have a non-zero root in common. Let $\mu_1$ and $\mu_2$ as before. Then by Lemma~\ref{lem:wronskian_vs_He}, $\He_{\mu_1}(x)$ and $\He_\lambda(x)$ have a non-zero root in common. As $R_{\mu_1}(x)$ is irreducible and shares a root with $R_\lambda(x)$, it must divide it. Then by Lemma~\ref{lem:shift_div}, $R_{\mu_1}(x)$ divides $R_{\mu_2}(x)$. However, from the proof of case (ii) we know that $R_{\mu_1}(x)$ can not share a root with $R_{\mu_2}(x)$, which gives the desired contradiction.

It follows that the degree sequence of $\lambda$ is semi-degenerate. Then we can apply Theorem~\ref{thm:gomez} to derive the conclusion of the corollary.
\end{proof}

We now move to the proof of Proposition~\ref{prop:staircase}. For any $n \geq k \geq 1$ define the function:
\begin{equation*}
\varphi_{(n,k)}(x)=\frac{Wr[\varphi_1(x),\varphi_3(x),\dots,\varphi_{2k-1}(x),\varphi_{n+k}(x)]}{Wr[\varphi_1(x),\dots,\varphi_{2k-1}(x)]}.
\end{equation*}
\begin{lemma}
\label{lem:properties_of_varphi}
Let $n \geq k \geq 1$.
\begin{enumerate}[(i)]
\item $Wr[\varphi_1(x),\dots,\varphi_{2k-1}(x)] = 2^{\binom{k}{2}+\binom{k+1}{2}}\Delta(1, 3, \ldots, 2k-1)x^{\binom{k+1}{2}}e^{\frac{-kx^2}{2}}$.
\item On $\mathbb{C} \setminus \{0\}$, the function $\varphi_{(n,k)}(x)$ is holomorphic and solves the differential equation 
\begin{equation}
\label{eq:phi_diff_eq}
\varphi_{(n,k)}''(x)=\left(x^2-(2n+1)+\frac{k(k+1)}{x^2}\right)\varphi_{(n,k)}(x).
\end{equation}
\item The degree sequence of the partition $\lambda=(n,k,k-1,\ldots,1)$ is semi-degenerate.
\end{enumerate}
\end{lemma}
\begin{proof}
The sequence $(1, 3, \ldots, 2k-1)$ is the degree sequence of the partition $\mu = (k, k-1, \ldots, 1)$. Then from Lemma~\ref{lem:wronskian_vs_He},
\begin{equation*}
Wr[\varphi_1(x),\dots,\varphi_{2k-1}(x)] = e^{\frac{-kx^2}{2}}2^{\binom{k}{2}+\frac{|\mu|}{2}}\Delta(1, 3, \ldots, 2k-1)\He_{\mu}(\sqrt{2}x).
\end{equation*}
But $\mu$ is a $2$-core, so $\He_{\mu}(x) = x^{|\mu|}$. Then $\He_{\mu}(\sqrt{2}x) = 2^{\frac{|\mu|}{2}}x^{|\mu|}$. Putting together these results shows (i).

For (ii), first note that by (i), $\varphi_{(n,k)}(x)$ is holomorphic on $\mathbb{C} \setminus \{0\}$. The fact that it verifies the differential equation \eqref{eq:phi_diff_eq} can be obtained by the same inductive argument described by Crum (\cite{Crum}). Crum proved this only for $\mathbb{R}$, but the argument extends without difficulty to the complex numbers. Therefore we do not reproduce it here.

For (iii), note that $\lambda$ has degree sequence $(1, 3, \ldots, 2k-1, n+k)$. Let $1 \leq i < j \leq k+1$. As $i \leq k$, the first Wronskian considered by the semi-degeneracy conditions only has the root $0$ by part (i) of this lemma. Then clearly the two Wronskians in the conditions can have at most the root $0$ in common. This shows that the degree sequence is semi-degenerate.
\end{proof}

\begin{proof}[Proof of Proposition~\ref{prop:staircase}]
The partition $\lambda$ has degree sequence $(1, 3, \ldots, 2k-1, n+k)$. From Lemma~\ref{lem:wronskian_vs_He}, it suffices to prove the statements for the roots of $\Wr[\varphi_{1}, \ldots, \varphi_{2k-1}, \varphi_{n+k}]$. Lemma~\ref{lem:properties_of_varphi}, (iii), ensures that we can apply Theorem~\ref{thm:gomez} to $\Wr[\varphi_{1}, \ldots, \varphi_{2k-1}, \varphi_{n+k}]$. Accordingly, the claims about the number and the simplicity of the non-zero real roots follow from this theorem.

Thus, if $n - k$ is odd, then by definition $d_\lambda = k+1$. Hence the multiplicity of the $0$ root is $\frac{(k+1)(k+2)}{2}$. Moreover, the total number of non-zero real roots is $n-k-1$. But then the number of real roots of $\He_\lambda(x)$ is $n-k-1 +\frac{(k+1)(k+2)}{2} = n+\frac{k(k+1)}{2} = |\lambda|$, so all the roots are real and those different from $0$ are simple.

Also, if $n - k$ is even, then by definition $d_\lambda = k-1$. Hence the multiplicity of the $0$ root is $\frac{(k-1)k}{2}$. Furthermore, the number of non-zero real roots is $n-k$. Since the total number of real roots is $n-k+\frac{(k-1)k}{2}$, it follows that $\He_\lambda(x)$ has exactly $2k$ complex non-real roots.

It remains to prove that in this case the non-real roots are simple.

Suppose for a contradiction that $z \neq 0$ is a root of multiplicity at least $2$. Then by Lemma~\ref{lem:properties_of_varphi}, (i), $z$ is also a root of multiplicity at least $2$ for $\varphi_{(n, k)}$. But then $\varphi_{(n, k)}$ verifies \eqref{eq:phi_diff_eq} at $z$, so $\varphi_{(n, k)}''(z)=0$. Differentiating the equation and using the fact that $\varphi_{(n, k)}'(z) = 0$ shows that $\varphi_{(n, k)}'''(z) = 0$. Continuing in this manner we obtain that all the derivatives of $\varphi_{(n, k)}$ at $z$ are $0$. As $\varphi_{(n, k)}$ is holomorphic around $z$, it must be identically $0$, a contradiction.
\end{proof}

The proof of Proposition~\ref{prop:staircase} relies crucially on the differential equation \eqref{eq:phi_diff_eq}. Writing $\varphi_{(n, k)}(x) = e^{-\frac{x^2}{2}}u_{(n, k)}(x)$, the function $u_{(n, k)}(x)$ verifies the differential equation:
\begin{equation}
\label{eq:u_diff_eq}
u''(x) - 2xu'(x) + \left(2n - \frac{k(k+1)}{x^2}\right)u(x) = 0.
\end{equation}
It is natural to ask about other solutions of \eqref{eq:u_diff_eq} that are holomorphic around $0$. It turns out these can be determined exactly.
\begin{proposition}
The differential equation \eqref{eq:u_diff_eq} has the two linearly independent solutions
\begin{equation*}
u_1(z) = a_0z^{k+1}\sum_{j=0}^{\infty}(-1)^j\frac{(n-k-1)(n-k-3)\ldots(n-k-2j+1)}{j!(2k+3)(2k+5)\ldots(2k+2j+1)}z^{2j}
\end{equation*}
and
\begin{equation*}
u_2(z) = b_0z^{-k}\sum_{j=0}^{\infty}\frac{(n+k)(n+k-2)\ldots(n+k-2j+2)}{j!(2k-1)(2k-3)\ldots(2k-2j+1)}z^{2j}.
\end{equation*}
If $n-k$ is odd, $u_1(z)$ equals $u_{(n, k)}(z)$ for a suitable choice of $a_0$, while if $n-k$ is even, $u_2(z)$ equals $u_{(n, k)}(z)$ for a suitable choice of $b_0$.
\end{proposition}
\begin{proof}
We put \eqref{eq:u_diff_eq} in the form
\begin{equation*}
u''(z) + \frac{f(z)}{z}u'(z) + \frac{g(z)}{z^2}u(z) = 0,
\end{equation*}
where $f(z) = -2z^2$ and $g(z) = 2nz^2 - k(k+1)$. Then $f(z)$ and $g(z)$ are analytic at $0$, so we can use the Frobenius method to determine two linearly independent solutions (\cite{Olver}, Chapter $5$). 

The indicial equation $\alpha(\alpha - 1) - k(k+1) = 0$ has the two solutions $\alpha_1 = k + 1$ and $\alpha_2 = -k$. These differ by an integer, so we start by determining the solution $u_1(z)$ for the largest root $\alpha_1$. Write $u_1(z) = z^{\alpha_1}\sum_{j=0}^\infty a_jz^j$.

The coefficients of $u_1(z)$ are given by the formula (\cite{Olver}, Chapter $5$, $(4.05)$):
\begin{equation*}
((\alpha_1 + j)(\alpha_1 + j - 1) - k(k+1))a_j = -\sum_{i=0}^{j-1}\frac{(\alpha_1+i)f^{(j-i)}(0) + g^{(j-i)}(0)}{(j-i)!}a_i.
\end{equation*}
As $f^{(j-i)}(0)$ and $g^{(j-i)}(0)$ are non-zero only for $i = j-2$, and $\alpha_1 = k + 1$, we obtain the recursive formula:
\begin{equation}
\label{eq:u1_coeff}
a_j = -\frac{1}{j(2\alpha_1 - 1 + j)}\frac{(\alpha_1+j-2)(-4) + 4n}{2}a_{j-2} = -\frac{2(n-k-j+1)}{j(2k + 1 + j)}a_{j - 2}.
\end{equation}
The first coefficient $a_0$ can be chosen arbitrarily. All the odd-indexed coefficients are $0$, while for $j \geq 1$, the even-indexed coefficients are
\begin{equation*}
a_{2j} = (-1)^j\frac{(n-k-1)(n-k-3)\ldots(n-k-2j+1)}{j!(2k+3)(2k+5)\ldots(2k+2j+1)}a_0.
\end{equation*}
This corresponds to the claimed expression for $u_1(z)$.

The series $u_1(z)$ terminates if and only if $n-k$ is odd. In this case $u_1(z)$ is a polynomial, and for a suitable choice of $a_0$ it equals $u_{(n, k)}(z)$.

According to \cite{Olver}, Chapter $5.5$, the second solution of the equation \eqref{eq:u_diff_eq} will be sought by reducing its order, hence let $u_2(z)=u_1(z) v(z)$. By replacing it in the equation \eqref{eq:u_diff_eq} we obtain that $v(z)$ has to satisfy the equation $\frac{v"(z)}{v'(z)}=2z-2\frac{u'_1(z)}{u_1(z)}$. By integrating this equation once, the result is 
\begin{equation}
\label{eq:varPhi_def}
v'(z)= C\frac{1}{(u_1(z))^2}e^{{z^2}} = \frac{\varPhi(z)}{z^{2k+2}}
\end{equation}
If we let $u_1(z)=z^{k+1}p(z)$, with $p(z)=\sum_{j=0}^\infty a_{2j} z^{2j}$, then there exist an analytic function $q(z)$ such that $\frac{1}{p^2(z)}=q(z)$. Moreover, since only even powers appear in the expression of $p(z)$, it will follow that the same is true for $q(z)$. Thus $\varPhi(z)=Cq(z)e^{{z^2}}$ is analytic and has only even powers in its analytic expression. 
The general solution for $v(z)$ is obtained by a new integration (see \cite{Olver}, Chapter $5.5$) and we get
\begin{equation*}
u_2(z) = u_1(z)v(z)= c_{2k+1}u_1(z)\ln(z) + z^{-k}\sum_{j=0}^\infty b_j z^j.
\end{equation*}
Now the coefficient $c_{2k+1}$ comes from the series development of the analytic function $\varPhi(z)$ appearing in \eqref{eq:varPhi_def}. Since in our case only the coefficients of the form $c_{2j}$ are different from $0$, it will follow that in fact $u_2(z) = z^{-k}\sum_{j=0}^\infty b_j z^j$. The coefficients of $u_2(z)$ verify the formula:
\begin{equation}
\label{eq:u2_coeff}
((\alpha_2 + j)(\alpha_2 + j - 1) - k(k+1))b_j = -\sum_{i=0}^{j-1}\frac{(\alpha_2+i)f^{(j-i)}(0) + g^{(j-i)}(0)}{(j-i)!}b_i.
\end{equation}
The first coefficient $b_0$ can be chosen arbitrarily. For $j \neq \alpha_1 - \alpha_2 = 2k+1$ we have
\begin{equation*}
b_j = -\frac{1}{j(j -2k - 1)}\frac{(-k+j-2)(-4) + 4n}{2}b_{j-2} = -\frac{2(n + k - j + 2)}{j(j-2k-1)}b_{j - 2}.
\end{equation*}
In particular, $b_j = 0$ if $j < 2k + 1$ and $j$ is odd. For $j = 2k + 1$, the left-hand side of \eqref{eq:u2_coeff} vanishes, so $b_{2k+1}$ can be chosen arbitrarily. The rest of the odd-indexed coefficients verify the same formula as \eqref{eq:u1_coeff}, that is $b_{2k+1+2j} = a_{2j}\frac{b_{2k+1}}{a_0}$ for $j \geq 0$. Therefore
\begin{equation*}
u_2(z) = b_0 z^{-k}\sum_{j=0}^\infty\frac{(n+k)(n+k-2)\dots (n+k-2j+2)}{j!(2k-1)(2k-3)\dots (2k-2j+1)}z^{2j} + \frac{b_{2k+1}}{a_0}u_1(z).
\end{equation*}
By choosing $b_{2k+1} = 0$ we obtain the claimed expression for $u_2(z)$.

Notice that in this case, the series $u_2(z)$ has a finite regular part if and only if $n-k$ is even. Then again for a suitable choice for $b_0$, we obtain $u_{(n, k)}(z)$.
\end{proof}

\section{Concluding remarks}

In this paper we have determined three instances in which the remainder polynomial $R_\lambda(x)$ is irreducible. We expect that there are other classes of partitions for which the remainder polynomial is irreducible.

In view of this, the results in this paper could be extended in the following directions. First, one could try to relax the requirements on $p$ in Theorem~\ref{thm:slope}, as well as find instances in which the upper bound can be sharpened. Lemma~\ref{lem:slope_for_3} is a step in this direction. Secondly, it would be very interesting to prove irreducibility for other partitions of length $2$. For unbalanced partitions $(n, m)$ where $n$ is substantially larger than $m$, Lemma~\ref{lem:high_degree_for_nm} rules out factors of high degree. Thus only the existence of small degree factors is still open.

In a different direction, it appears the full potential of Theorem~\ref{thm:gomez} has not been exploited yet. We have used it to determine the number of real roots of $\He_\lambda(x)$ for $\lambda = (n, k, k-1, \ldots, 1)$. It would be interesting to find further applications of this theorem.

Finally, in Proposition~\ref{prop:staircase} we have established Veselov's conjecture for $(n, k, k-1, \ldots, 1)$. It is likely there are other partitions for which Veselov's conjecture can be shown to hold.

\section*{Acknowledgements}

We would like to thank Professor Nair for sending us a copy of \cite{NairShorey16}.

\bibliographystyle{abbrv}
\bibliography{bibl}
\end{document}